\documentclass[11pt]{article}

\usepackage{latexsym,amsfonts}
\usepackage{amssymb,amsthm,upref,amscd}
\usepackage[T1]{fontenc}
\usepackage{times}
\usepackage{amscd}
\usepackage{mathrsfs}
\usepackage[reqno]{amsmath}
\usepackage[numbers,sort&compress]{natbib}
\usepackage{braket}
\newtheorem{theorem}{Theorem}
\newtheorem{lemma}{Lemma}[section]
\newtheorem{proposition}[lemma]{Proposition}
\newtheorem{corollary}[lemma]{Corollary}
\newtheorem{remark}[lemma]{Remark}

\newtheorem*{notation}{Notation}
\renewcommand{\theequation}{\arabic{section}.\arabic{equation}}

\makeatletter \@addtoreset{equation}{section} \makeatother

\begin{document}

\title{\bf Bound states for logarithmic Schr\"odinger equations with
 potentials unbounded below
\thanks{E-mail: cxzhang@amss.ac.cn (C. Zhang), darkblue1121@163.com (X. Zhang).}}

\author{Chengxiang Zhang$^{\rm a}$, \; Xu Zhang$^{\rm b}$}

\date{\small\it$^{\rm a}$Academy of Mathematics and Systems Science, Chinese Academy of Sciences, Beijing 100190, China\\
$^{\rm b}$School of Mathematics and Statistics, Central South University,
Changsha 410083,
  China}
\maketitle

\begin{minipage}{13cm}
{\small {\bf Abstract:} We study the existence and concentration behavior of the bound states for the following logarithmic Schr\"odinger equation
\begin{equation*}
\begin{cases}
-\varepsilon^2\Delta v+V(x)v=v\log v^2 \  \ &\text {in}\ \ \mathbb R^N,\\
v(x)\to 0 \ \ &\text {as}\ \ |x|\to\infty,
\end{cases}
\end{equation*}
where $N\geq 1$, $\varepsilon>0$ is a small parameter, and $V$ may be  unbounded below at infinity with  a speed of  at most quadratic strength.
We show that around various types of  local topological   critical points of the potential function, positive bound state solutions exist and concentrate as $\varepsilon\to0$.


}
\end{minipage}

\section{Introduction and main Results}

 In this paper, we study the   semiclassical states of the following
logarithmic Schr\"odinger equation
\begin{equation}\label{1.1}
\begin{cases}
-\varepsilon^2\Delta v+V(x)v=v\log v^2 \  \ &\text { in}\ \ \mathbb R^N,\\
v(x)\to 0 \ \ &\text { as}\ \ |x|\to\infty,
\end{cases}
\end{equation}
where $N\geq 1$ and $\varepsilon>0$ is a small parameter.
The problem comes from the study of standing waves to the time-dependent Schr\"odinger equation with logarithmic nonlinearity (\cite{BM1,BM2})
\begin{equation}\label{1.2}
i\hbar\frac{\partial \psi}{\partial t}+{\hbar^2}\Delta\psi-M(x)\psi+\psi\log |\psi|^2=0,
\end{equation}
where $\hbar$ denotes the Plank constant, $i$ is the imaginary unit. We call $\psi$ a standing wave solution if it possesses the form $\psi(x,t)=\exp\{-iEt/\hbar\}v(x)$.
 Then $\psi$ is a standing wave solution for \eqref{1.2} if and only if $v$ solves
 \eqref{1.1}
with $\varepsilon^2=\hbar$ and $V(x)=M(x)-E$.
For the dimensionless logarithmic Schr\"odinger equations, i.e., \eqref{1.2} with $\hbar=1$, standing waves
 have been studied in recent years in \cite{dAvenia2014,JS,Tanaka-Zhang,Tr,Wang-Zhang-1}.
 In these papers, multiple existence,  uniqueness and asymptotic behaviors of bound state solutions  are studied
 for \eqref{1.1} (for $\varepsilon=1$)
 with various potential functions which are bounded from below.
 For the semiclassical states of logarithmic Schr\"odinger equations, as $\varepsilon\to 0$, \cite{AdMF-18,AJ-19,AJ2} studied the existence of positive solutions  of \eqref{1.1} which are localized around global minimum points or global saddle points of the potential.
 The authors of \cite{Wang-Zhang-2} constructed an unbounded sequence of sign-changing bound state solutions around a local minimum point of the potential.
 In \cite{ITWZ}, the potential function of \eqref{1.1} is assumed  to possess a finite number of singularities of at most logarithmic strength and localized bound  state solutions are constructed around the singular points.
 We note that these results for semiclassical states of logarithmic Schr\"odinger equations are motivated  by
 the extensive study of  the semiclassical Schr\"odinger equation with power-law nonlinearity:
\begin{equation}\label{1.3}
\begin{cases}
-\varepsilon^2\Delta v+V(x)v=|v|^{p-2}v \  \ &\text { in}\ \ \mathbb R^N,\\
v(x)\to 0 \ \ &\text { as}\ \ |x|\to\infty,
\end{cases}
\end{equation}
where $p\in (2,2^*)$ with
$2^*=2N/(N-2)$ if $N\geq 3$ and $2^*=\infty$ if
$N\leq 2$.
Starting from the pioneer work \cite{FW} and \cite{Rabinowitz1992}, there have been a great deal of work on the existence of semiclassical states for \eqref{1.3}.
 See \cite{AR,AW,BV,BV2,Byeon2002,BW-1,BW-2,BW,Schaftingen2013,DelPino-Felmer1996,DelPino-Felmer1997,WX1,WX} and the reference therein for more discussion on \eqref{1.3}.

The results mentioned above consider \eqref{1.1} with a general condition that potential functions are bounded from below, while in the present paper we investigate logarithmic Schr\"odinger equations with non-confining potentials which may be unbounded below and propose a variational framework to tackle this case.
 As a meaningful example, we point that the following logarithmic Schr\"odinger equation
\begin{equation}\label{ex}
	-\varepsilon^2\Delta u-|x|^2u=u\log u^2
\end{equation}
admits at least a positive bound state solution for each $\varepsilon\in(0, \frac12)$, with its explicit formula $\exp\{\frac{1+\sqrt{1-4\varepsilon^2}}4(N-\varepsilon^{-2}|x|^2)\}$, which is a single peak solution  localized around the origin as $\varepsilon\to0$.
So it is essential to propose the general  conditions on potential function (with the case $-|x|^2$ included) which ensure the existence of solutions to equation \eqref{1.1}.
We also refer to \cite{BF,CC,EB,ITWZ,WW} for more discussion
on Schr\"odinger operators or Schr\"odinger equations with potentials unbounded below.
To study positive bound state solutions to \eqref{1.1} in this situation, we state the precise assumptions on $V$.
Assume that
\begin{enumerate}
	\item[(V0)] $\liminf_{|x|\to\infty}V(x)|x|^{-2}>-\infty$;
\item[(V1)] $V\in C(\mathbb R^N,\mathbb R)$ and there is a bounded domain
  $\Omega\subset {\mathbb{R}}^{N}$  such that $$V_0:=\min\limits_{x\in \Omega}V(x)<\min\limits_{x\in \partial\Omega}V(x).$$
\end{enumerate}
Under the assumptions (V0) and (V1),
$$\mathcal V:=\set{ x \in \Omega |   V(x) =V_0}$$   is a nonempty compact subset of $\Omega$.
Without loss of generality, we can assume that $\partial\Omega$ is smooth
and $0\in\mathcal V\subset\Omega$.
Throughout this paper, for any set $\Lambda\subset\mathbb R^N$, $\delta>0$, $\varepsilon>0$,
we denote
\begin{equation*}
\begin{split}
&\Lambda^\delta=\set{x\in\mathbb R^N | {\rm dist}(x, \Lambda):=\inf_{y\in\Lambda}|x-y|<\delta},\\
&\Lambda_\varepsilon=\set{x\in\mathbb R^N | \varepsilon x\in\Lambda}.
\end{split}
\end{equation*}
Then our first result is as follows.
\begin{theorem}\label{th1.1}
Suppose (V0) and (V1) hold. Then there exists $\varepsilon_0>0$ such that for each $\varepsilon\in(0,\varepsilon_0)$, equation \eqref{1.1} admits a positive solution $v_\varepsilon$ satisfying
\begin{itemize}
\item[(i)] $v_\varepsilon$ possesses a unique local maximum point $x_\varepsilon\in \Omega$
 such that $\text{dist} (x_\varepsilon,\mathcal V)\to 0$ and
 $$v_\varepsilon(\varepsilon x+x_\varepsilon)\to v(x)\quad \text{strongly in}\quad H^1(\mathbb R^N) \quad \text{as}\quad \varepsilon\to 0,
 $$
 where
$v(x)=\exp\{\frac{N+V_0-|x|^2}2\}$ is the unique positive radial solution of
\begin{equation}\label{eq1.4}
	-\Delta v+V_0v= v\log v^2,\quad v\in H^1(\mathbb R^N).
\end{equation}
\item[(ii)] for any $\delta>0$, there exist  $C,c>0$ such that
$$v_\varepsilon(x)\leq Ce^{-c\varepsilon^{-2}|x-x_\varepsilon|^2}\quad \text{for}\ \ x\in\mathbb R^N.$$
\end{itemize}
\end{theorem}

Clearly, the existence results in Theorem \ref{th1.1} can not cover the example given in \eqref{ex}.
In fact, the solution to \eqref{ex} given in closed form is localized around the maximum point of $-|x|^2$ while Theorem \ref{th1.1} deals with local minimum case.
To establish a general result which include the existence of a solution to equation \eqref{ex}, we adopt the following assumptions from \cite{DelPino-Felmer1997} and \cite{DelPino-Felmer2002}, which cover several classes of general critical points of potential function including local maximum and saddle point situation.

\begin{enumerate}
\item[(V2)] $V\in C^1(\mathbb R^N,\mathbb R)$ and there is an open and bounded set $\Omega$ with smooth boundary and  closed nonempty sets $\mathcal B,\mathcal B_0$ of $\Omega$ such that $\mathcal B$ is connected and $\mathcal B_0\subset \mathcal B$. Moreover,
\begin{equation}\label{eq1.9}
\mu_0:=\inf_{\gamma\in\mathcal T}\max_{x\in \mathcal B} V(\gamma(x))>\sup_{x\in \mathcal B_0}V(x),
\end{equation}
where $\mathcal T=\set{\gamma\in C(\mathcal B,\Omega) | \gamma(x)=x\ \text{for each}\ x\in \mathcal B_0}$.
\item[(V3)]  For any $x\in\partial \Omega$ such that $V(x)=\mu_0$,  $\partial_T V(x)\neq 0$, where $\partial_T $ denotes the tangential derivative.
\end{enumerate}
Then our second result is as follows.
\begin{theorem}\label{th1.2}
Suppose (V0), (V2) and (V3) hold. Then there exists $\varepsilon_0>0$ such that for each $\varepsilon\in(0,\varepsilon_0)$, equation \eqref{1.1} admits a positive solution $v_\varepsilon$ satisfying \begin{itemize}
\item[(i)] $v_\varepsilon$ possesses a unique local maximum point $x_\varepsilon\in \Omega$
 such that $V(x_\varepsilon)\to \mu_0$, $\nabla V(x_\varepsilon)\to 0$ and
 $$v_\varepsilon(\varepsilon x+x_\varepsilon)\to v(x)\quad \text{strongly in}\quad H^1(\mathbb R^N) \quad \text{as}\quad \varepsilon\to 0,
 $$
 where
$v(x)=\exp\{\frac{N+\mu_0-|x|^2}2\}$ is the unique positive radial solution of
\begin{equation*}
	-\Delta v+\mu_0v= v\log v^2,\quad v\in H^1(\mathbb R^N).
\end{equation*}
\item[(ii)] for any $\delta>0$, there exist  $C,c>0$ such that
$$v_\varepsilon(x)\leq Ce^{-c\varepsilon^{-2}|x-x_\varepsilon|^2}\quad \text{for}\ \ x\in\mathbb R^N.$$
\end{itemize}
\end{theorem}
\begin{remark}\label{rek1.1}
For $\Omega$ in (V1) or (V2), we can fix $R_0>0$ such that $\Omega \subset B(0,R_0/2)$.
Moreover,  substituting $v$ with $\lambda v$ in \eqref{1.1} for a proper constant $\lambda>0$, we may  assume without loss of generality that
\begin{equation}\label{1.5}
V(x)\geq 1,\ \ \text { for\ }x\in B(0,R_0).
\end{equation}
\end{remark}
The study is motivated by a series of work on vanishing potential problems in semiclassical Schr\"odinger equations with power-law nonlinearity \eqref{1.3}.
Some general  assumptions on potentials which appear in this problem are  $\liminf_{x\to\infty}V(x)=0$ and $\inf_{\mathbb R^N}V(x)=0$.
In \cite{AFM},
Ambrosetti et al. studied the equation with  potential vanishing slowly at infinity and having positive local minimum.
Moroz and Van Schaftingen in \cite{Schaftingen2010} weakened the assumptions on the decaying rate of potential at infinity, including in particular the case that the potential possesses compact support. We note that the condition
$\inf_{x\in\mathbb R^N}V(x)=0$ is first introduced in \cite{Byeon2002} as the critical frequency case
for  \eqref{1.3},
for the reason that if $\inf_{x\in\mathbb R^N}V(x)<0$, neither ground state solutions nor nice limit problems exist as $\varepsilon\to 0$.
For other related results, we refer the readers to \cite{AR,AW,BV,BV2,BW-1,BW-2,BW,Schaftingen2013}.
On the other hand, the condition $\inf_{x\in\mathbb R^N}V(x)=0$ is no longer critical for existence of solutions to the logarithmic Schr\"odinger equations \eqref{1.1}.
In fact, it has been shown in \cite{ITWZ} that even the potential possesses several singular points at which $V\to-\infty$ with a speed of up to the logarithmic strength, there exist bound states with small amplitude concentrated around these singularities.
This expresses a different profile of logarithmic type equations.
To further understand the difference, it is worthwhile to investigate on another general case that $\liminf_{|x|\to\infty}V(x)=-\infty$.

Comparing with the results of  Schr\"odinger equations with power-law nonlinearities, as well as those with logarithmic nonlinearity in the literature,
the main novelty in our results is that the potential $V(x)$ may tend to $-\infty$ at infinity. To explain the difficulties
in our setting, set $u(x)=v(\varepsilon x)$ in \eqref{1.1}. Then
 \eqref{1.1} is equivalent to
\begin{equation}\label{eq2}
-\Delta u+V(\varepsilon x)u=u\log u^2 \  \ \text { in}\ \ \mathbb R^N,
\end{equation}
which is the Euler--Lagrange equation
associated with the energy functional:
\begin{equation*}
J_{\varepsilon}(u)=\frac{1}{2}\int_{\mathbb R^N}\big(|\nabla u|^2+V(\varepsilon x)u^2\big)dx
-\frac{1}{2}\int_{\mathbb R^N}\big(u^2\log u^2-u^2\big)dx.
\end{equation*}
We note that if $V$  is bounded from below, by  rescaling $w=\lambda u$ in \eqref{eq2}, the equation can be shifted to another one with a positive-definite linear part (\cite[Remark 1.1]{Tanaka-Zhang}), which enables one to use common variational approaches,
such as the  constraint minimization methods or the minimax principle,
to search for critical levels of the corresponding functional in an appropriate functional space.
However, under assumption (V0), the spectrum of operator $-\Delta +V(\varepsilon x)$ may still be unbounded below.
Therefore, the rescaling mentioned above does not work and it is difficult to seek out a mountain pass structure
for the corresponding functional.
On the other hand, the non-compactness problem is also noticeable since a large class of  strongly repulsive potential functions is included by our assumptions.
Besides,  the non-smoothness of functional
$J_\varepsilon$ in $H^1(\mathbb R^N)$
caused by the special growth of  logarithmic nonlinearity near $0$ is an additional difficulty.
To overcome these difficulties, for $R_0$ fixed in Remark \ref{rek1.1}, we truncate $\widetilde V(x)=\max\{V(x),|x|^2\}$ in $\mathbb R^N\setminus B(0,R_0)$ and consider a modified functional defined on the weighted Sobolev space:
$$H_\varepsilon=\Set{u\in H^1(\mathbb R^N) | \int_{\mathbb R^N }\widetilde V(\varepsilon x)u^2<\infty}.$$
Here we note that the functional $\int_{\mathbb R^N}u^2\log u^2$ is well-defined and $C^1$ smooth at any $u\in H_\varepsilon$ (see Lemma \ref{lem 2.1}).
We state the formula and some properties of the  functional in brief as follows. For each $u\in H_\varepsilon$, set
\begin{equation*}
\bar J_{\varepsilon}(u)=\frac{1}{2}\int_{\mathbb R^N}\big(|\nabla u|^2+\widetilde V(\varepsilon x)u^2\big)dx
-\frac{1}{2}\int_{\mathbb R^N}\big(u^2\log u^2-u^2\big)dx+\Psi_\varepsilon(u),
\end{equation*}
 where $\Psi_\varepsilon$ is a non-positive valued functional defined in \eqref{Psi} possessing the following properties:
 \begin{itemize}
 	\item Both $|\Psi_\varepsilon(u)|$ and $\|\Psi_\varepsilon'(u)\|_{H_\varepsilon^{-1}}$ are infinitesimal as $\varepsilon\to 0$ uniformly for $u\in H_\varepsilon$.

 \item If $u(x)\leq \exp\{-\varepsilon|x|^2\}$ in $\mathbb R^N\setminus B(0,R_0\varepsilon^{-1})$,
 then $\bar J_\varepsilon(u)=J_\varepsilon(u)$ and $\bar J'_\varepsilon(u)=J'_\varepsilon(u)$.
 \end{itemize}

With this modification, we can deal with a $C^1$ smooth problem defined on $H_\varepsilon$ which embeds to the Lebesgue space $L^q(\mathbb R^N)$ ($\frac{2N}{N+2}<q<2^*$)  compactly.
Moreover, by the first property of $\Psi_\varepsilon$, the modified functional $\bar J_\varepsilon$ possesses  mountain-pass or   linking structures for small $\varepsilon$. In this way, we can find a critical point when $\varepsilon$ is small.
 However, this critical point is not necessarily the solution for equation \eqref{eq2}.
 In order to recover the original problem \eqref{eq2}, it is important to obtain the localization and decay property for this candidate by the second property of $\Psi_\varepsilon$.
 To get over this problem, we
introduce another penalization on the
nonlinearity (essentially the penalization of del Pino and Felmer \cite{DelPino-Felmer1996,DelPino-Felmer1997})
and turn to the study of critical points of
$$\Gamma_\varepsilon(u)=\frac{1}{2}\int_{\mathbb R^N}\big(|\nabla u|^2+\widetilde V(\varepsilon x)u^2\big)dx
-\frac{1}{2}\int_{\mathbb R^N}F_\varepsilon(x,u)dx+\Psi_\varepsilon(u).
$$
See Section \ref{sec2} for the exact expression of  $F_\varepsilon(x,u)$.
The penalization argument can help   localize critical points of $\Gamma_\varepsilon$ around $\Omega_\varepsilon$.
Then the final step  to get a solution is to prove the desired decay property of the critical point for small $\varepsilon$.
 We would like to mention that the exponential decay estimate for solutions to logarithmic  Schr\"odinger equations is made in \cite{ITWZ,Wang-Zhang-2}.
   However, it is not applicable to recover our original problem
   by the second property of $\Psi_\varepsilon$.
   We will achieve a uniform Gaussian decay estimate for these critical points by making use of the properties of
   $\Psi_\varepsilon$ and  the singular nature of the logarithmic terms.
   With the decay estimate  established, we are able to obtain a solution for the original problem.
\begin{remark}
We note that in \cite{AdMF-18,AJ-19,AJ2}, the authors studied the existence of positive solutions to semiclassical logarithmic equation with bounded potential which possesses  global minima or global saddle points.
The results therein are confined on some global assumptions on the potential function.
We point out that  our assumption covers more general cases, especially at infinity the potential subjects to a very weak restriction.
\end{remark}

Our method is rather robust and works for more general situations.
As an extension, we consider logarithmic Schr\"odinger equation with competing potentials
 \begin{equation}\label{1.4}
 \begin{cases}
 -\varepsilon^2\Delta v+V(x)v=K(x)v\log v^2 \  \ &\text { in}\ \ \mathbb R^N,\\
 v(x)\to 0 \ \ &\text { as}\ \ |x|\to\infty.
 \end{cases}
\end{equation}
We note that semiclassical states for power-law type Schr\"odinger equation with competing potentials is first studied in \cite{WX}, where the existence of ground states as well as the  concentration behavior are proved for small $\varepsilon$. See \cite{AFM,BV2,Schaftingen2010} for more discussions on the power-law type Schr\"odinger equations with vanishing completing potentials. To state our result, we make the following assumptions:
 \begin{itemize}
 \item[(P1)] $V, K\in C(\mathbb R^N,\mathbb R)$, $K\in L^\infty(\mathbb R^N)$, $K(x)>0$, and there is a bounded domain
  $\Omega\subset {\mathbb{R}}^{N}$ such that
   $$P_0:=\min\limits_{x\in \Omega}P(x)<P_1:=\min\limits_{x\in \partial\Omega}P(x),\quad
   P(x):=|K(x)|^{-\frac{N}{2}+1}e^{\frac{V(x)}{K(x)}}.$$
 \item[(P2)] There exist $\mu\geq0$ and $\kappa\geq 0$ with $\frac{1}{2}\mu+\kappa\leq 1$ such that
      $$\liminf_{|x|\to\infty}V(x)|x|^{-\mu}>-\infty,\quad \liminf_{|x|\to\infty}K(x)|x|^{\kappa}>0.$$
 \end{itemize}
 Obviously, the set of minimum points for $P$ in $\Omega$
 $$\mathcal P=\set{ x \in \Omega |   P(x) =P_0}\neq\emptyset$$   is compact in $\Omega$.
 And we have the  following theorem.
 \begin{theorem}\label{th2}
 Let (P1) and (P2) hold. Then there exists $\varepsilon_0>0$ such that for all $\varepsilon\in(0,\varepsilon_0)$, equation \eqref{1.4} has a positive solution $v_\varepsilon$ satisfying
 \begin{itemize}
 	\item[(i)] for any $\delta>0$, there exist   $C,c>0$ such that
 $$v_\varepsilon(x)\leq Ce^{-c\varepsilon^{-2}(\text{dist}(x, \mathcal P^\delta))^{2-\kappa}}\quad \text{for}\quad x\in \mathbb R^N;$$
 \item[(ii)] up to a subsequence, there exist $\varepsilon_k\to 0$, $x_{\varepsilon_k}$ and $x_0\in \mathcal P$ with $x_{\varepsilon_k}\to x_0$
 such that
 $$v_{\varepsilon_k}(\varepsilon_k x+x_{\varepsilon_k})\to v(x)\quad \text{as}\ k\to \infty\ \text{strongly in}\ H^1(\mathbb R^N),
 $$  where
 $v(x)=\exp\{\frac{V(x_0)}{2K(x_0)}+\frac{N-K(x_0)|x|^2}2\}$ is the unique positive radial solution of
 $$-\Delta v+V(x_0)v= K(x_0)v\log v^2,\quad v\in H^1(\mathbb R^N).$$
 \end{itemize}
 \end{theorem}

\begin{remark}
A vanishing potential $K$ in the nonlinearity cause additional difficulties in the proof of recovering the original problem since it affects the decay rate of the solution at infinity.
\end{remark}

We shall give more extensions.
i) We consider \eqref{1.1} with potential possessing a finite number of singular points, and prove the existence of nontrivial solutions which  concentrate around these singular points.
ii) If we make some assumptions on the derivatives of potential function on the boundary of the domain $\Omega$, we can obtain an unbounded sequence of  bound state solutions as $\varepsilon\to 0$.
These results, which generalize the results in \cite{ITWZ,Wang-Zhang-2} to the case where potential may be unbounded from below at infinity, will be given as an appendix.

The remainder of this paper is organized as follows. In Section \ref{sec2}, we define the auxiliary function, derive a variational setting for the problem, introduce the modified functional by a penalization approach  and give the solutions for the modified problem.
In Section 3, we prove that  if $V$ has a local minimum, the solutions for the modified problem are indeed solutions for the original equation when $\varepsilon$ is small.
In Section 4, we prove Theorem \ref{th1.2}.
   In Section 5, we generalize the result of Theorem \ref{th1.1} to  equations with competing potentials and prove   Theorem \ref{th2}.
   At last in Appendix, we give some more extensions on the results in Theorem 1.1.

\begin{notation} Throughout this paper,
$2^*=+\infty$ for $N=1,2$ and $2^*=\frac{2N}{N-2}$ for $N\geq 3$; $ L^p(\mathbb R^N) \ (1\leq p<+\infty)$ is the usual Lebesgue space with the norm $ |u| _p^p=\int_{\mathbb R^N}|u|^p;$
$ H^1(\mathbb R^N)$ denotes the Sobolev space with the norm $\|u\|^2=\int_{\mathbb R^N}(|\nabla u|^2+|u|^2);$
$o_n(1)$ (resp. $o_\varepsilon(1)$)
will denote a generic infinitesimal as $n\rightarrow \infty$ (resp. $\varepsilon\rightarrow 0^+$);  $B (x,\rho)$ denotes an open ball centered at $x\in\mathbb R^N$ with radius
$\rho>0$.
$a^\pm=\max\{0,\pm a\}$ for $a\in\mathbb R$.
 Unless stated otherwise, $\delta$ and $C$ are general constants.
\end{notation}

\section{Modified problem and preliminaries}\label{sec2}
Throughout this section, we assume that $V\in C(\mathbb R^N)$ satisfies (V0)  and there are bounded domain $\Omega$ and $R_0>0$ such that $0\in \Omega \subset B(0,R_0/2)$  and $V\geq 1$ in $B(0, R_0)$.

We first define the auxiliary function.
Fixing $\phi_\varepsilon(x)=\exp\{-\varepsilon|x|^2\}$, we set
\begin{equation}\label{eta}
\eta_\varepsilon(x,t)=
\left\{
\begin{aligned}
&0,\  &t&\in [0, \phi_\varepsilon(x)]\cup [5\phi_\varepsilon(x),\infty);\\
&-\frac{(t-\phi_\varepsilon(x))^2}{4\phi_\varepsilon^3(x)},
\ & t&\in[\phi_\varepsilon(x), 2\phi_\varepsilon(x)];\\
&\frac{(t-3\phi_\varepsilon(x))^2}{4\phi_\varepsilon^3(x)}-\frac{1}{2\phi_\varepsilon(x)},
\ & t&\in[2\phi_\varepsilon(x), 4\phi_\varepsilon(x)];\\
&-\frac{(t-5\phi_\varepsilon(x))^2}{4\phi_\varepsilon^3(x)},
\ & t&\in [4\phi_\varepsilon(x), 5\phi_\varepsilon(x)].
\end{aligned}
\right.
\end{equation}
It is clear that $\eta_\varepsilon(x,t)$ is  $C^1$ continuous in  $t\in[0,\infty)$ and
\begin{equation}
\begin{split}
|\eta_\varepsilon'(x,t)|\leq   \frac12 \phi_\varepsilon(x)^{-2},
 \end{split}
\end{equation}
where $\eta_\varepsilon'(x,t)$ denotes the partial derivative of $\eta_\varepsilon(x,t)$ relative to $t$.
 Hence
 $\widehat\eta_\varepsilon(x,t)$ defined blow is  $C^2$ in $t\in[0,\infty)$:
\begin{equation}\label{hateta}
\widehat\eta_\varepsilon(x,t):=1+\int_0^t\eta_\varepsilon(x,s)ds.
\end{equation}
Moreover, $\widehat\eta_\varepsilon\in [0,1]$, $\widehat\eta_\varepsilon(x,t)=1$ for $0\leq t\leq \phi_\varepsilon(x)$ and $\widehat\eta_\varepsilon(x,t)=0$ for $t\geq 5\phi_\varepsilon(x)$.
To penalize the nonlinearity, we also introduce  functions
\begin{equation*}\label{g}
 g(s)=
\begin{cases}
2e^{-1}, \ &s\leq -e^{-1},\\
s\log s^2,\ & |s| \leq e^{-1},\\
-2e^{-1}, \ &s\geq e^{-1},
\end{cases}
\ \ \text{and} \  \ G(s)=\int_{0}^sg(t)dt
=\begin{cases}
2e^{-1}s+\frac12e^{-2}, \ &s\leq -e^{-1},\\
\frac12s^2\log s^2-\frac12s^2,\ & |s| \leq e^{-1},\\
-2e^{-1}s+\frac12e^{-2}, \ &s\geq e^{-1}.
\end{cases}
\end{equation*}
We note that $g\in C(\mathbb R)\cap C^1(\mathbb R\setminus\{0\})$ is odd and
$G\in C^1( \mathbb R )\cap C^2(\mathbb R\setminus\{0\})$ is even.
The following lemma gives some direct properties of the auxiliary functions which will be used frequently in the subsequent argument.
\begin{lemma}\label{lem etaq}
For $x\in\mathbb R^N$ and $s\in\mathbb R$, the following statements hold
\begin{itemize}
\item[(i)] $\widehat\eta_\varepsilon(x,|s|)s^2\leq 25\phi_\varepsilon^2(x)$,
$\big|\eta_\varepsilon(x,|s|)s^3\big|\leq 125\phi_\varepsilon^2(x)$;
\item[(ii)] $\widehat\eta_\varepsilon(x,|s|)|s|\leq 5\phi_\varepsilon(x)$,
 $\big|\eta_\varepsilon(x,|s|)s^2\big|\leq 25\phi_\varepsilon(x)$, $\big|\eta_\varepsilon'(x,|s|)s^3\big|\leq 125\phi_\varepsilon(x)$;
\item[(iii)] $g(s)\leq \min\{s\log s^2, -2e^{-1}\}$ if $s\geq0$, $-\frac12(s^2\log s^2)^--2s^2\leq G(s)\leq0$ if $s\in \mathbb R$;
 \item[(iv)] $g(s)=s\log s^2$ and
 $G(s)=\frac{1}{2}s^2\log s^2-\frac{1}{2}s^2$ if $|s|\leq e^{-1}$;
 \item[(v)] $g(s)s-2G(s)\in C^1(\mathbb R)$ is such that  $0\leq g(s)s-2G(s)\leq s^2$ if $|s|\geq e^{-1}$ and
 $g(s)s-2G(s)= s^2$ if $|s|\leq e^{-1}$.
\end{itemize}
\end{lemma}
Let $\widetilde V:\mathbb R^N\rightarrow[1,+\infty)$ be a  function such that
\begin{equation}\label{tildev}
\widetilde V(x)=
\begin{cases}
V(x),\ &|x|< R_0;\\
\max\{V(x),|x|^2\}, \ &|x|\geq R_0.
\end{cases}
\end{equation}
For $\widetilde V_\varepsilon(x)=\widetilde V(\varepsilon x)$, $V_\varepsilon(x)=V(\varepsilon x)$ and
$\overline V_\varepsilon =V_\varepsilon -\widetilde V_\varepsilon$,
 we define $\Psi_\varepsilon$ by
\begin{equation}\label{Psi}
\Psi_\varepsilon(u)= \frac{1}{2}\int_{\mathbb R^N} \overline V_\varepsilon (x) \widehat\eta_\varepsilon(x,|u|)u^2dx.
\end{equation}
Note that by Lemma \ref{lem etaq} (ii) and \eqref{Psi}, $\Psi_\varepsilon$ is well-defined on the Hilbert space
\begin{equation}\label{H}
H_\varepsilon:=\Set{u\in H^1(\mathbb R^N) | \int_{\mathbb R^N}\widetilde V_\varepsilon(x)u^2dx<\infty},
\end{equation}
with inner product $(u,v)_\varepsilon:=\int_{\mathbb R^N}\nabla u \nabla v+\widetilde V_\varepsilon (x)uv$ and norm $\|u\|_\varepsilon:=\sqrt{(u,u)_\varepsilon}$.
 Moreover, for any $u,v\in H_\varepsilon$, there holds
\begin{equation}\label{2.1}
\Psi_\varepsilon'(u)v =\int_{\mathbb R^N} \overline V_\varepsilon (x)\big(\frac{1}{2}\eta_\varepsilon(x,|u|)|u|u+\widehat\eta_\varepsilon(x,|u|)u\big)vdx.
\end{equation}

\begin{corollary}\label{cor2.3}
 For some $C,c>0$ independent of $\varepsilon$, there holds
\begin{align*}
\sup_{u\in H_\varepsilon}|\Psi_\varepsilon(u)|+
 \sup_{u\in H_\varepsilon}\|\Psi_\varepsilon'(u)\|_{H_\varepsilon^{-1}}+\sup_{u\in H_\varepsilon}\|\Psi_\varepsilon''(u)u\|_{H_\varepsilon^{-1}}\leq Ce^{-c\varepsilon^{-1}},
\end{align*}
where  $\|\cdot\|_{H_\varepsilon^{-1}}$denotes the norm on the dual space of $H_\varepsilon$.
\end{corollary}
\begin{proof} By \eqref{Psi}, \eqref{2.1} and Lemma \ref{lem etaq}, it suffices to notice that
\begin{equation}\label{2.4}
\int_{| x|\geq R_0\varepsilon^{-1}} |x|^2\phi_\varepsilon^2\leq Ce^{-c\varepsilon^{-1}}
\end{equation}
holds for some $C,c>0$ independent of $\varepsilon$.
\end{proof}
Let $\chi$ be the characteristic function of $\mathbb R^N \setminus \Omega$ and set $\chi_\varepsilon(x)=\chi(\varepsilon x)$. Denote
\begin{align*}
f_\varepsilon(x,s)&:=(1-\chi_\varepsilon(x))s\log s^2+\chi_\varepsilon(x)g(s)=(1-\chi_\varepsilon(x))(s\log s^2-g(s))+g(s),\\
F_\varepsilon(x,s)&:=\int_0^sf_\varepsilon(x,t)dt=\frac12(1-\chi_\varepsilon(x))(s^2\log s^2-s^2)+\chi_\varepsilon(x)G(s).
\end{align*}
Define the functional:
\begin{equation}\label{2.3}
\Gamma_{\varepsilon}(u)=\frac{1}{2}\|u\|_\varepsilon^2+\Psi_\varepsilon(u)
-\int_{\mathbb R^N}F_\varepsilon(x,u),\quad u\in H_\varepsilon.
\end{equation}
We note that $\Gamma_{\varepsilon}$ is well defined and of class $C^1$ on $H_\varepsilon$,
which follows from the following lemma due to $\widetilde V(x)\geq |x|^2$ for $|x|\geq R_0$:

\begin{lemma}\label{lem 2.1}
For any $q\in (\frac{2N}{N+2},2)$, there exists a constant $C>0$ independent of $\varepsilon\in(0,1)$, such that
 $|u|_q\leq C\varepsilon^{(\frac12-\frac1q)N}\|u\|_\varepsilon$, $u\in H_\varepsilon$. Moreover, $H_\varepsilon$ embeds compactly into $L^q(\mathbb R^N)$ for any $q\in(\frac{2N}{N+2},2^*)$. In particular, $u_n\rightharpoonup u$ weakly in $H_\varepsilon$ implies that
$$\int_{\mathbb R^N}u_n^2\log u_n^2\to \int_{\mathbb R^N}u^2\log u^2,\quad \int_{\mathbb R^N}G(u_n)\to \int_{\mathbb R^N}G(u), \quad\text{as}\ n\to\infty.$$
\end{lemma}
\begin{proof}
For $q\in (\frac{2N}{N+2},2)$ and $u\in H_\varepsilon$, we have
$$\begin{aligned}
\int_{\mathbb R^N} |u|^q &=\int_{|x|\leq R_0\varepsilon^{-1}} |u|^q+\int_{|x|\geq R_0\varepsilon^{-1}} |u|^q\\
&\leq C\varepsilon^{\frac{q-2}2N}|u|_2^q+
\Big(\int_{|x|\geq R_0\varepsilon^{-1}} |\varepsilon x|^2|u|^2\Big)^\frac{q}2\Big(\int_{|x|\geq R_0\varepsilon^{-1}} |\varepsilon x|^{-\frac{2q}{2-q}}\Big)^\frac{2-q}2\\
&\leq C\varepsilon^{\frac{q-2}2N}\|u\|_\varepsilon^q.
\end{aligned}
$$
Then it follows that $|u|_q\leq C\varepsilon^{(\frac12-\frac1q)N}\|u\|_\varepsilon$. And the compact embedding holds since $\widetilde V_\varepsilon(x)\to+\infty$ as $|x|\to\infty$.
\end{proof}
We note that any critical point of $\Gamma_\varepsilon$ is a solution to
\begin{equation}\label{eq2.8}
-\Delta u+T(x,u)u= f_\varepsilon(x,u),\quad x\in \mathbb R^N,
\end{equation}
where
$$T(x,u)=\widetilde V_\varepsilon(x)+\overline V_\varepsilon (x)\big(\frac{1}{2}\eta_\varepsilon(x,|u|)|u|+\widehat\eta_\varepsilon(x,|u|)\big).
$$

\begin{remark}\label{rek2.9}
(i) By Kato's inequality, any solution $u\in H_\varepsilon$ to \eqref{eq2.8}
weakly solves
\begin{equation}\label{eq2.13}
-\Delta |u|+|u|+\overline V_\varepsilon (x)\widehat\eta_\varepsilon(x,|u|)|u|\leq |u|\log u^2,\quad x\in \mathbb R^N,
\end{equation}
since $\widetilde V\geq 1$, $\overline V\leq 0$, $\eta_\varepsilon\leq 0$ and $f_\varepsilon(x,s)\leq s\log s^2$ for $s\geq 0$.
Furthermore, by the definition of $\overline V$ and Lemma \ref{lem etaq} (ii),
$\overline V_\varepsilon (x)\widehat\eta_\varepsilon(x,s)=0$ if $|\varepsilon x|< R_0$ and
$|\overline V_\varepsilon (x)\widehat\eta_\varepsilon(x,s)s|\leq C e^{-c\varepsilon|x|^2}$ if $|\varepsilon x|\geq R_0$ for  some $c, C>0$ independent of $\varepsilon$.
So for some $p\in(2,2^*)$, $u$ solves
\begin{equation}\label{eq2.9}
-\Delta |u|+|u| \leq C(|u|^{p-1}+\mathbf{1}_{R_0,\varepsilon}e^{-c\varepsilon |x|^2}),\quad x\in \mathbb R^N,
\end{equation}
where $\mathbf{1}_{R_0,\varepsilon}$ denotes the characteristic function of
$\mathbb R^N\setminus B(0,R_0\varepsilon^{-1})$.

(ii) For a family of $\{w_\varepsilon\}\subset H_\varepsilon$ with $\|w_\varepsilon\|_\varepsilon\geq \varepsilon^2$ satisfying $\Gamma_\varepsilon'(w_\varepsilon)w_\varepsilon=0$, it is standard to show that
\begin{equation}\label{eq3.14}
\liminf_{\varepsilon\to 0}\|w_\varepsilon\|_{\varepsilon}>0,\quad \liminf_{\varepsilon\to 0}|w_\varepsilon|_{p}>0,
\end{equation}
for $p\in (2,2^*)$.
In fact, similar to the argument of \eqref{eq2.9}, we  can prove  for some $C, c>0$,
$$\|w_\varepsilon\|^2_\varepsilon\leq \int_{\mathbb R^N} |w_\varepsilon|^p+\mathbf{1}_{R_0,\varepsilon}e^{-c\varepsilon|x|^2}|w_\varepsilon|).
$$
Then by Sobolev inequality,
\begin{equation*}
 |w_\varepsilon|_p^2\leq C\|w_\varepsilon\|^2_\varepsilon
\leq C(|w_\varepsilon|_{p}^{p}+e^{-c/\varepsilon}|w_\varepsilon|_p).
\end{equation*}
So we have either $\lim_{\varepsilon\to 0}|w_\varepsilon|_p>0$ or $|w_\varepsilon|_p\leq Ce^{-c/\varepsilon}$ for small $\varepsilon$. Then \eqref{eq3.14} holds since $\|w_\varepsilon\|_\varepsilon\geq \varepsilon^2$.

\end{remark}

Next we give more properties about $\Gamma_\varepsilon$.

\begin{lemma}\label{lem 2.3}
 Assume $M\in (0,\infty)$, $\varepsilon\in(0,1)$ and $u_\varepsilon\in H_\varepsilon$ satisfy
$$\Gamma_\varepsilon(u_\varepsilon)\leq M,\ \ \ |\Gamma_\varepsilon'(u_\varepsilon)u_\varepsilon|\leq M\|u_\varepsilon\|_\varepsilon.$$
Then there hold
$$\|u_\varepsilon\|_\varepsilon\leq C(M),\quad \int_{\mathbb R^N} |G(u_\varepsilon)|\leq C(M),
\quad
\int_{\mathbb R^N} \big|u_\varepsilon^2\log u_\varepsilon^2\big|\leq C(M),$$
 for some constant $C(M)>0$ independent of $\varepsilon$.
\end{lemma}
\begin{proof}
Using  Lemma \ref{lem etaq} (v) and Corollary \ref{cor2.3}, it is easy to see
\begin{equation*}
\begin{split}
M(1+\|u_\varepsilon\|_\varepsilon)&\geq \Gamma_\varepsilon(u_\varepsilon)-\frac{1}{2}\Gamma_\varepsilon'(u_\varepsilon)u_\varepsilon
         \geq-Ce^{-c/\varepsilon}
         +\frac{1}{2}\int_{\Omega_\varepsilon\cup \{|u_\varepsilon|\leq e^{-1}\}} |u_\varepsilon|^2.
\end{split}
\end{equation*}
Therefore, we have $\|u_\varepsilon\|_{L^2(\Omega_\varepsilon\cup \{|u_\varepsilon|\leq e^{-1}\})}^2\leq C(M)(1+\|u_\varepsilon\|_\varepsilon)$, $\varepsilon\in(0,1)$. So by Gagliardo--Nirenberg inequality we have
\begin{equation*}
\begin{split}
\int_{\Omega_\varepsilon\cup \{|u_\varepsilon|\leq e^{-1}\}} \big(u_\varepsilon^2\log u_\varepsilon^2\big)^+
&\leq C_0\int_{\Omega_\varepsilon\cup \{|u_\varepsilon|\leq e^{-1}\}}|u_\varepsilon|^{p_0}
\leq C_0'\|u_\varepsilon\|_{L^2(\Omega_\varepsilon\cup \{|u_\varepsilon|\leq e^{-1}\})}^{p_0(1-\theta)}\|u_\varepsilon\|_\varepsilon^{p_0\theta}\\
&\leq C_0'\big[C(M)(1+\|u_\varepsilon\|_\varepsilon)\big]^{\frac{p_0}{2}(1-\theta)}\|u_\varepsilon\|_\varepsilon^{p_0\theta},
\end{split}
\end{equation*}
where $2<p_0<2^*,\theta=\frac{N(p_0-2)}{2p_0}$, $C_0$ and $C'_0$ are constants depending only on $p_0$ and $N$.
Choosing $p_0=2+\frac{2}{N+2}$, we have
\begin{equation}\label{2.10}
\begin{split}
\int_{\Omega_\varepsilon\cup \{|u_\varepsilon|\leq e^{-1}\}} \big(u_\varepsilon^2\log u_\varepsilon^2\big)^+ \leq C'(M)(1+\|u_\varepsilon\|_\varepsilon)^{\frac{3}{2}}.
\end{split}
\end{equation}
This, together with $\Gamma_\varepsilon(u_\varepsilon)\leq M$, $\Psi_\varepsilon(u_\varepsilon)\geq -C$ for $\varepsilon\in(0,1)$ and $G\leq 0$, leads us to the fact
\begin{equation*}
\begin{split}
M+C'(M)(1+\|u_\varepsilon\|_\varepsilon)^{\frac{3}{2}}+C
&\geq
\frac{1}{2}\|u_\varepsilon\|_\varepsilon^2
+\frac{1}{2}\int_{\Omega_\varepsilon\cup \{|u_\varepsilon|\leq e^{-1}\}}\big(u_\varepsilon^2\log u_\varepsilon^2\big)^-,
\end{split}
\end{equation*}
which implies
$$\|u_\varepsilon\|_\varepsilon\leq C(M)\quad
\text{and}\quad
\int_{\Omega_\varepsilon\cup \{|u_\varepsilon|\leq e^{-1}\}}\big(u_\varepsilon^2\log u_\varepsilon^2\big)^-\leq C(M).
$$
And then the conclusion follows from \eqref{2.10} and
$\int_{\{|u_\varepsilon|\geq e^{-1}\}}\big(u_\varepsilon^2\log u_\varepsilon^2\big)^-
\leq 2\int_{\mathbb R^N}u_\varepsilon^2.$
\end{proof}

\begin{corollary}\label{cor 2.4}
For $\varepsilon\in(0,1)$, $\Gamma_\varepsilon$ satisfies Palais--Smale condition.
\end{corollary}
\begin{proof}
Fix $\varepsilon\in(0,1)$, let $\{u_n\}$ be a Palais--Smale sequence for $\Gamma_\varepsilon$. According to Lemma \ref{lem 2.3}, up to a subsequence, $u_n\rightharpoonup u$ weakly in $H_\varepsilon$ for some $u\in H_\varepsilon$, and hence by Lemma \ref{lem 2.1}
$u_n\to u$ strongly in $L^q(\mathbb R^N)$, $\frac{2N}{N+2}<q<2^*$ and
$$\int_{\mathbb R^N}f_\varepsilon(x,u_n)u_n\to \int_{\mathbb R^N}f_\varepsilon(x,u)u,\ \ n\to\infty.$$
Moreover, by dominated convergence theorem and Lemma \ref{lem etaq},
$\Psi_\varepsilon'(u_n)u_n\to \Psi_\varepsilon'(u)u$ as $n\to\infty$. So from
$\lim_{n\to\infty}\Gamma_\varepsilon'(u_n)u_n=0$ and $\Gamma_\varepsilon'(u)u=0$, it follows
$\lim_{n\to\infty}\|u_n\|_\varepsilon^2=\|u\|_\varepsilon^2$. Then the conclusion follows.
\end{proof}

The next corollary gives a uniform sub-solution estimate for the critical point of $\Gamma_\varepsilon$.
\begin{corollary}\label{lem 2.7}
For $M>0$, let $u_\varepsilon$ be critical point
of $\Gamma_\varepsilon$ with $\Gamma_\varepsilon(u_\varepsilon)\leq M$.
Then for any $\rho\in(0,1)$, there exists a
constant $C=C(M,\rho)>0$, independent of $\varepsilon$ and $x\in\mathbb R^N$, such that
\begin{equation*}
\begin{aligned}
&|u_\varepsilon(x)|\leq C\|u_\varepsilon\|_{L^2(B(x,\rho))},
& &x\in B(0,R_0\varepsilon^{-1}-\rho),\\
&|u_\varepsilon(x)|\leq C(\|u_\varepsilon\|_{L^2(B(x,\rho))}+\varepsilon),
& &x\in\mathbb R^N.
\end{aligned}
\end{equation*}
\end{corollary}
\begin{proof}
Note that  by  Lemma \ref{lem 2.3}, $\|u_\varepsilon\|_{H^1}\leq C(M)$. By Remark \ref{rek2.9}, $u_\varepsilon$ satisfies \eqref{eq2.9}.
Then by a standard iteration procedure, there holds $|u_\varepsilon|_{\infty}\leq C(M)$, and by the sub-solution estimates in \cite{Trudinger1997} (see also \cite{S}), one can prove the conclusion.
\end{proof}

Since we have assumed that $0\in\Omega$, there holds $B(0,1)\subset\Omega_\varepsilon$ for small $\varepsilon$.
Fix $\omega \in C_0^\infty(B(0,1))\setminus\{0\}$.
We can verify that $\Gamma_\varepsilon$ possesses a mountain-pass structure.
\begin{lemma}\label{lem 2.5}
There exist positive constants $\varepsilon_0, t_0, r_0, M_0$ and $M_1$ such that for $\varepsilon\in(0,\varepsilon_0)$, the following statements hold.
\begin{itemize}
\item[(i)]
$\sup_{t\geq t_0}\Gamma_\varepsilon(t\omega)<-2$  and
 $\sup_{t\geq 0}\Gamma_\varepsilon(t\omega)\leq M_1.$
\item[(ii)]
 $\inf_{\|u\|_\varepsilon=r_0}\Gamma_\varepsilon(u)\geq M_0$  and $\inf_{\|u\|_\varepsilon\leq r_0}\Gamma_\varepsilon(u)\geq -1$.
\end{itemize}
\end{lemma}
\begin{proof}
For small $\varepsilon_0$ and $\varepsilon\in(0,\varepsilon_0)$, there holds $B(0,1)\subset\Omega_\varepsilon$. Thus for any $x\in\mathbb R^N$ and $t\in (0,\infty)$, we have $f_\varepsilon(x, t\omega(x))=t\omega(x)\log (t\omega(x))^2$ . So
\begin{equation}\label{2.5}
\begin{split}
\Gamma_\varepsilon(t\omega)&=\frac{t^2}{2}\int_{\mathbb R^N}|\nabla \omega|^2
+(\widetilde V_\varepsilon(x)+1)\omega^2-\omega^2\log (t^2\omega^2)\\
&=\frac{t^2}{2}\int_{\mathbb R^N}|\nabla \omega|^2
+(\widetilde V_\varepsilon(x)+1)\omega^2-\omega^2\log \omega^2
-\frac{t^2\log t^2}{2}\int_{\mathbb R^N}\omega^2.
\end{split}
\end{equation}
Since $\widetilde V_\varepsilon(x)$ is uniformly bounded on $B(0,1)$, there exists $t_0>0$
independent of $\varepsilon\in(0,\varepsilon_0)$ such that $\Gamma_\varepsilon(t\omega)<0$ for $t\geq t_0$, which  implies further that $\sup_{t\geq 0}\Gamma_\varepsilon(t\omega)\leq M_1$  for some $M_1>0$.

To prove (ii), we notice that
 $u^2\log u^2\leq C_0|u|^{p_0}$ for $ 2<p_0<2^*$.
 Therefore by  Corollary \ref{cor2.3} and Sobolev inequality,
$$\Gamma_\varepsilon(u)\geq \frac{1}{2}\|u\|_\varepsilon^2-C_0'\|u\|_\varepsilon^{p_0}-Ce^{-c\varepsilon^{-1}}.$$
It follows that
there exist $r_0$ and $M_0>0$ independent of $\varepsilon$ such that
$$\Gamma_\varepsilon(u)\geq -Ce^{-c\varepsilon^{-1}},\ \|u\|_\varepsilon\leq r_0\ \  \text { and} \ \ \Gamma_\varepsilon(u)\geq 2M_0-Ce^{-c\varepsilon^{-1}},\ \ \|u\|_\varepsilon=r_0.$$
Then making $\varepsilon_0$ smaller if necessary,  (ii)   follows.
\end{proof}

In order to find critical points of $\Gamma_\varepsilon$, for each $\varepsilon\in (0,\varepsilon_0)$, define the mountain pass value
 for  the modified functional
$\Gamma_\varepsilon$
\begin{equation}\label{eq2.12}
d_\varepsilon=\inf_{h\in \mathcal H_\varepsilon}\max_{s\in[0,1]}\Gamma_\varepsilon(h(s)),
\end{equation}
where
$$\mathcal H_\varepsilon:=\set{h\in C([0,1],H_\varepsilon) |  h(0)=0,\ \Gamma_\varepsilon(h(1))<-2}.$$
Let $\varepsilon_0$, $M_0$ and $M_1$ be the positive constants fixed in Lemma \ref{lem 2.5}. Then we have
\begin{proposition}\label{pro 2.6}
For each $\varepsilon\in(0,\varepsilon_0)$,  $\Gamma_{\varepsilon}$ possesses a nontrivial critical point $u_\varepsilon\in H_\varepsilon$
 satisfying
$\Gamma_\varepsilon(u_\varepsilon)=d_\varepsilon\in[M_0,M_1]$.
Moreover, $u_\varepsilon$ is a positive weak solution to \eqref{eq2.8}.
\end{proposition}
\begin{proof}
By Lemma \ref{lem 2.5},
for each fixed $\varepsilon\in (0,\varepsilon_0)$,
$d_\varepsilon\in [M_0, M_1]$. Let $\{h_n\}\subset \mathcal H_\varepsilon$ satisfy $\max_{t\in [0,1]}\Gamma_\varepsilon(h_n(t))\to d_\varepsilon$ as $n\to\infty$. Since  $\Gamma_\varepsilon(u)=\Gamma_\varepsilon(|u|)$ for each $u\in H_\varepsilon$, we have $|h_n|\in \mathcal H_\varepsilon$ and $\max_{t\in [0,1]}\Gamma_\varepsilon(|h_n(t)|)\to d_\varepsilon.$
By Lemma \ref{lem 2.5} and the  minimax principle (see \cite[Theorem 2.8]{willem}),
we can find $\{u_n\}\in H_\varepsilon$ and $\{t_n\}\subset[0,1]$ such that as $n\to\infty$,
$$\Gamma_\varepsilon(u_n)\to d_\varepsilon,\quad
\|\Gamma_\varepsilon'(u_n)\|_{H_\varepsilon^{-1}}\to 0\quad \text{and}\quad \|u_n-|h_n(t_n)|\|_\varepsilon\to 0.
$$
Thus by Corollary \ref{cor 2.4} and $|h_n(t_n)|\geq 0$, there is a nontrivial critical point $u_\varepsilon\in \mathcal H_\varepsilon$ of $\Gamma_\varepsilon$ such that
 $u_n\to u_\varepsilon$ in $H_\varepsilon$ and $u_\varepsilon\geq0$.
By the maximum principle in \cite{Va}, $u_\varepsilon>0$.
\end{proof}

To get more information about the energy level, we  recall some results for the  autonomous logarithmic Schr\"odinger equation,
which is related to the limit problem for \eqref{eq2}.
Up to translations in $\mathbb R^N$, the equation
\begin{equation*}
\begin{cases}
-\Delta v=v\log v^2 \  \ &\text { in}\ \ \mathbb R^N,\\
v(x)\to 0 \ \ &\text { as}\ \ |x|\to\infty,
\end{cases}
\end{equation*}
possesses a unique positive solution
$U(x)=e^{\frac{N}{2}}e^{-\frac{1}{2}|x|^2}.$
Note that for $a\in  \mathbb R$, $U_a(x):=e^{\frac{a}{2}}U(x)$ is the unique positive solution (up to translations) to the equation
\begin{equation*}
\begin{cases}
-\Delta v+a v=v\log v^2 \  \ &\text { in}\ \ \mathbb R^N,\\
v(x)\to 0 \ \ &\text { as}\ \ |x|\to\infty.
\end{cases}
\end{equation*}
$U_a$ is the ground state of the corresponding functional
$$I_a(u)=\frac{1}{2}\int_{\mathbb R^N}|\nabla u|^2+(a+1) u^2
-u^2\log u^2,\quad u\in H^1(\mathbb R^N).$$
That is to say, the following minimizing problem
\begin{equation}\label{3.2}
m(a):=\inf_{u\in\mathcal N_a}I_a(u),
\end{equation}
where
$$\begin{aligned}
\mathcal N_a:&=\Set{u\in \mathcal D\setminus\{0\} | \int_{\mathbb R^N}|\nabla u|^2+a u^2-u^2\log u^2=0},\\
 \mathcal D:&=\Set{u\in H^1(\mathbb R^N)| \int_{\mathbb R^N}|u^2\log u^2|<\infty},
\end{aligned}
$$
is achieved by $U_a$.
It is easy to check that
\begin{equation}\label{3.1}
m(a)=\max_{t\geq0}I_a(tU_a)=I_a(U_a)=\frac{1}{2}\int_{\mathbb R^N}|U_a|^2=\frac{e^a}{2}|U|_2^2,
\end{equation}
 which is a strictly increasing function of $a\in\mathbb R$.
 We refer to \cite{dAvenia2014,Tr} for more information on the unique positive solution $U(x)$.

\begin{lemma}\label{lem2.10}
 Let $V_1\in \mathbb R$, $\varepsilon_n\to 0$, $v_{\varepsilon_n}\in H_{\varepsilon_n}$  be such that
\begin{equation}\label{eq2.21}
\lim_{n\to\infty}\Gamma_{\varepsilon_n}(v_{\varepsilon_n})= m(V_1), \quad \lim_{n\to\infty}\|\Gamma_{\varepsilon_n}'(v_{\varepsilon_n})\|_{H_{\varepsilon_n}^{-1}}=0,
\end{equation}
 Then necessarily $V_1\geq \inf_{x\in \Omega}V(x)$ and $m(V_1)\geq m(\inf_{x\in \Omega}V(x))$.
Moreover, if we assume further that
 $V_1<\inf_{x\in\Omega}V(x)+\log 2,$ then
 there are $y_n\in \mathbb R^N$ and $x_0\in \set{x\in\overline \Omega|V(x)\leq V_1})$
  such that, up to a subsequence,  $\varepsilon_ny_n\to x_0$
 and
  $v_{\varepsilon_n}(\cdot+y_n)$ converges to   $U_{V(x_0)}$ weakly in $H^1(\mathbb R^N)$  and strongly in $L^p(\mathbb R^N)$ for $p\in(2,2^*)$.
\end{lemma}

\begin{proof}
 For clarity, we write $\varepsilon=\varepsilon_n$.
 By Lemma \ref{lem 2.3}, we know for some constant $C_1>0$ independent of $\varepsilon$,
\begin{equation}\label{3.3}
\|v_\varepsilon\|_{\varepsilon}\leq C_1,\quad
\int_{\mathbb R^N} \big(v_\varepsilon^2\log v_\varepsilon^2\big)^-\leq C_1.
\end{equation}
On the other hand, by \eqref{3.1}--\eqref{3.3}, Lemma \ref{lem etaq} (v) and Corollary \ref{cor2.3}, we have
$$\begin{aligned}
0<2m(V_1)=\lim_{\varepsilon\to0}\big(2\Gamma_\varepsilon(v_\varepsilon)-\Gamma_\varepsilon'(v_\varepsilon)v_\varepsilon\big)
=&\lim_{\varepsilon\to0}\int_{\mathbb R^N}\big((1-\chi_\varepsilon)v_\varepsilon^2+\chi_\varepsilon(g(v_\varepsilon)v_\varepsilon-2G(v_\varepsilon))\big)\\
\leq& \liminf_{\mathbb R^N}\int_{\mathbb R^N} v_\varepsilon^2\leq \liminf_{\varepsilon\to0}\|v_\varepsilon\|_\varepsilon^2.
\end{aligned}
$$
Hence, for each fixed $p\in (2,2^*)$, there is
$C_p>0$ such that,
 $$\liminf_{\varepsilon\to0}\int_{\mathbb R^N}|v_\varepsilon|^p\geq C_p\liminf_{\varepsilon\to0}\int_{\mathbb R^N}f_\varepsilon(x,v_\varepsilon)v_\varepsilon
=C_p\liminf_{\varepsilon\to0}\big(\|v_\varepsilon\|_\varepsilon^2-\Gamma_\varepsilon'(v_\varepsilon)v_\varepsilon\big)>0.$$
Then, by  P.L. Lions' lemma (\cite{lions}), there is $y_\varepsilon\in \mathbb R^N$ such that \begin{equation}\label{eq2.19}
 \liminf_{\varepsilon\to0}
 \int_{B(y_\varepsilon,1)}|v_\varepsilon|^2>0.
 \end{equation}
Up to a subsequence if necessary, we assume
$v_\varepsilon(x+y_\varepsilon)\rightharpoonup v\neq 0$
weakly in $H^1(\mathbb R^N)$.
By Fatou's Lemma
and \eqref{3.3}, $v\in \mathcal D$.
If $\lim_{\varepsilon\to 0}\text{dist}( y_\varepsilon,\Omega_\varepsilon)\to\infty$, then
$\lim_{\varepsilon\to0}\chi_\varepsilon(x+y_\varepsilon)= 1$ for each $x\in \mathbb R^N$.
 By  Corollary \ref{cor2.3}, $\widetilde V\geq1$ and
$\lim_{\varepsilon\to0}\Gamma_\varepsilon'(v_\varepsilon)v(\cdot-y_\varepsilon)=0,$
we obtain
$$\int_{\mathbb R^N}|\nabla v|^2+v^2 \leq \int_{\mathbb R^N}g(v)v\leq 0.$$
Therefore, $\limsup_{\varepsilon\to 0}\text{dist}( y_\varepsilon,\Omega_\varepsilon)<\infty$ and especially,  $\lim_{\varepsilon\to 0}\text{dist}(\varepsilon y_\varepsilon,\Omega)=0$.

Then up to a subsequence
we may assume
$\varepsilon y_\varepsilon\to x_0\in \overline \Omega$.
By Corollary \ref{cor2.3}, \eqref{eq2.21} and \eqref{3.3}, it is easy to check that $v\in \mathcal D$ is a solution to
\begin{equation}\label{eq2.23}
-\Delta v+V(x_0)v=(1-\widetilde\chi)v\log v^2+\widetilde \chi g(v),
\end{equation}
where $0\leq \widetilde \chi(x) \leq 1$ is the limit function of $\chi_{\varepsilon}(x+y_\varepsilon)$, which is identically $0$  if ${\rm dist}(y_\varepsilon,\partial\Omega_\varepsilon)\to\infty$, or otherwise the characteristic function of the half space
$$H=\set{y\in\mathbb R^N | y\cdot \vec n(x_0)\geq0},$$
with
$\vec n (x_0)$  the outward normal vector to $\partial \Omega$ at $x_0$.
 We consider  the functional corresponding to \eqref{eq2.23} in $\mathcal D$:
\begin{equation*}
\widetilde I(u)=\frac{1}{2}\int_{\mathbb R^N}|\nabla u|^2+V(x_0)u^2
-\int_{\mathbb R^N}\frac{1}{2}(1-\widetilde\chi )\big(u^2\log u^2-u^2\big)+\widetilde\chi G(u).
\end{equation*}
By Lemma \ref{lem etaq} (iv),
$G(s)-\frac12(s^2\log s^2-s^2)\leq 0$.
Then by the monotonicity of $g(s)/s$, we can check that
$$\widetilde I(v)=\max_{t\geq 0}\widetilde I(tv)\geq \max_{t\geq 0} I_{V(x_0)}(tv)\geq m(V(x_0)).$$
Now by Lemma \ref{lem etaq} (v), Corollary \ref{cor2.3} and Fatou's lemma, we have
\begin{equation}\label{3.9}
\begin{split}
m(V_1)=\liminf_{\varepsilon\to0}\Gamma_{\varepsilon}(v_\varepsilon)
&=\liminf_{\varepsilon\to0} \Big(\Gamma_{\varepsilon}(v_\varepsilon)-\frac{1}{2}\Gamma_{\varepsilon}'(v_\varepsilon)v_\varepsilon\Big)\\
&\geq \liminf_{\varepsilon\to0}  \frac12\int_{\mathbb R^N}\Big((1-\chi_\varepsilon )v_\varepsilon^2+\chi_\varepsilon(g(v_\varepsilon)v_\varepsilon-2G(v_\varepsilon))\Big)
\\
&\geq
\frac{1}{2}\int_{\mathbb R^N} \Big((1-\widetilde \chi)v^2+\widetilde \chi(g(v)v-2G(v))\Big)\\
&=\widetilde I(v)\geq m(V(x_0)).
\end{split}
\end{equation}
Then
$m(V_1)\geq m(\inf_{x\in\Omega}V(x))$
and by the  monotonicity of  $m(\cdot)$, $V_1\geq \inf_{x\in\Omega}V(x)$.

Next we assume further that  $V_1<\inf_{x\in\Omega}V(x)+\log 2$.
To proceed, we  claim that the weak limit $v$ does not change sign. Otherwise,
we can check that
$$\widetilde I(v)=\max_{t\geq 0}\widetilde I(tv^+)+\max_{t\geq 0}\widetilde I(tv^-)\geq 2 m(V(x_0))
=e^{V(x_0)}|U|_2^2>\frac{e^{V_1}}2|U|_2^2=m(V_1),$$
which contradicts to $\widetilde I(v)\leq m(V_1)$
by \eqref{3.9}.
We remark also that $v\in W^{2,p}(\mathbb R^N)\cap C^{1,\sigma}(\mathbb R^N)$ for all  $p\in(1,\infty)$ and $\sigma\in(0,1)$ by the regularity theory.

If $\widetilde\chi\equiv0$,   there directly holds
\begin{equation}\label{2.26}
-\Delta v+V(x_0)v=v\log v^2.
\end{equation}
For the case that $\widetilde\chi$ is the characteristic function of  $H$, we test \eqref{eq2.23} by $\nabla v\cdot\vec n(x_0)$ and integrate on $\mathbb R^N$.
Noting that $-\vec n(x_0)$ is the outward unit normal vector to $\partial H$, by divergence theorem,  we have
$$ \int_{\partial H} v^2\log v^2-v^2-2G(v)=0.$$
Since $s^2\log s^2-s^2-G(s)\geq 0$, the above formula implies $v^2\log v^2-v^2=2G(v)$ on $\partial H$ and hence
$v\leq e^{-1}$ on $\partial H$.
Noting also that $-\Delta v+V(x_0)v=\widetilde\chi g(v)\leq 0$ in $H$, we can apply the maximum principle to obtain $v\leq e^{-1}$ in $H$ and
thus  $g(v)=v\log v^2$ in $H$. Therefore,
 $v$ weakly solves \eqref{2.26} with $m(V(x_0))\leq I_{V(x_0)}(v)\leq m(V_1)$, which implies $V(x_0)\leq V_1$. Moreover, by the maximum principle in \cite{Va}, $v>0$.

Next we show $|v_\varepsilon(\cdot +y_\varepsilon)- v|_{p}\to0$. If not, up to a subsequence, we assume $\lim_{\varepsilon\to0}|v_\varepsilon(\cdot +y_\varepsilon)- v|_{p}> 0$ for some $p\in (2,2^*)$. Then
we can find another sequence of $y_\varepsilon^1\in\mathbb R^N$
satisfying $|y_\varepsilon^1-y_\varepsilon|\to\infty$, $\varepsilon y_\varepsilon^1\to x_1\in \overline \Omega$ and
\eqref{eq2.19}  for $y_\varepsilon^1$.
 Therefore, $v_\varepsilon(\cdot +y_\varepsilon^1)\rightharpoonup v_1\neq 0$, where $v_1>0$ solves $-\Delta v+V(x_1)v=v\log v^2$.
Then similarly to \eqref{3.9},
we can check that
\begin{equation*}
\begin{split}
\liminf_{\varepsilon\to0}\Gamma_{\varepsilon}(v_\varepsilon)
&\geq \liminf_{\varepsilon\to0}  \frac12\int_{\mathbb R^N}\Big((1-\chi_\varepsilon )v_\varepsilon^2+\chi_\varepsilon(g(v_\varepsilon)v_\varepsilon-2G(v_\varepsilon))\Big)
\\
&\geq
\frac{1}{2}\int_{\mathbb R^N} v^2+v_1^2= m(V(x_0))+m(V(x_1))\geq 2m(\inf_{x\in \Omega}V(x))>m(V_1),
\end{split}
\end{equation*}
which is a contradiction. Therefore,
$|v_\varepsilon(\cdot +y_\varepsilon)- v|_{p}\to0$.
 Replacing $y_\varepsilon$ by $y_\varepsilon+y$ with $y$ the maximum point of $v$, we can assume without loss of generality that  $v= U_{V(x_0)}$. This completes the proof.
\end{proof}

\section{Proof of Theorem \ref{th1.1}}\label{sec3}
In this section, we assume (V0) and (V1). Without loss of generality, we can assume that $\partial\Omega$ is smooth,
 $0\in\mathcal V\subset\Omega\subset B(0,R_0/2)$ for some $R_0>0$ and $V\geq 1$  in  $B(0,R_0).$
 Then by Proposition \ref{pro 2.6}, $\Gamma_\varepsilon$ defined in Section \ref{sec2} has a critical point $u_\varepsilon>0$ for $\varepsilon\in (0,\varepsilon_0)$.
 To prove it is actually a solution to the original problem for small $\varepsilon$, we first estimate the upper energy   bounds.

\begin{lemma}\label{lem 3.1}
Let $d_\varepsilon$ be defined in \eqref{eq2.12}. Then
$$\limsup_{\varepsilon\to0}d_\varepsilon\leq m(V_0),$$
 where $m(\cdot)$ is the function defined in \eqref{3.2}.
\end{lemma}
\begin{proof}
Let $u_0=e^{\frac{V_0}{2}}U$. Similarly to \eqref{2.5}, one can find $t_0>0$ such that
$I_{V_0}(tu_0)<-4$ for $t\geq t_0$. Then for any $k\geq 1$,
there is $u_k\in C_0^\infty(\mathbb R^N)$ such that $\sup_{t\in[0,t_0]}|I_{V_0}(tu_0)-I_{V_0}(tu_k)|\leq 1/k.$
In particular, $I_{V_0}(t_0u_k)<-3$. Since $u_k$ has compact support, when $\varepsilon$ is small enough, there holds $\Psi_\varepsilon(tu_k)=0$, $f_\varepsilon(x,tu_k)=tu_k\log(tu_k)^2$ for $t\geq 0$, and hence
$$\Gamma_\varepsilon(tu_k)=I_{V_0}(tu_k)
+\frac{1}{2}\int_{\mathbb R^N} (\widetilde V_\varepsilon(x)-V_0)|tu_k|^2.$$
Therefore $\lim_{\varepsilon\to0}\Gamma_\varepsilon(tu_k)=I_{V_0}(tu_k)$ uniformly holds for $t\in[0,t_0]$, which leads us to the fact that
 $\sup_{t\in[0,t_0]}|\Gamma_{\varepsilon}(tu_k)-I_{V_0}(tu_k)|\leq 1/k$
and in particular $\Gamma_\varepsilon(t_0u_k)<-2$ when $\varepsilon$ is small.
For $s\in[0,1]$, set $\gamma_k(s):=st_0u_k$.
Then for $\varepsilon$ small, there holds $\gamma_k\in\mathcal H_{\varepsilon}$,
which implies by \eqref{3.1},
$$\limsup_{\varepsilon\to0}d_\varepsilon
\leq \limsup_{\varepsilon\to0}\sup_{s\in[0,1]}\Gamma_\varepsilon(\gamma_k(s))
\leq \sup_{t\in[0,t_0]}I_{V_0}(tu_0)+2/k
\leq m(V_0)+2/k.$$
Then the conclusion follows from the arbitrary choice of $k$.
\end{proof}

Next, we shall focus on the localization of $u_\varepsilon$.
Lemma \ref{lem 3.1} and the strict monotonicity of $m(\cdot)$ will ensure that the solution $u_\varepsilon$ is localized around the set $\mathcal V$ when $\varepsilon$ is small.

\begin{lemma}\label{lem 3.2} Let $u_\varepsilon$ be obtained in Proposition \ref{pro 2.6}. Then for any $\delta>0$, there holds
$$\lim_{\varepsilon\to0}\|u_\varepsilon\|_{L^\infty(\mathbb R^N\setminus (\mathcal V^\delta)_\varepsilon)}=0.$$
\end{lemma}
\begin{proof}
By  Lemma \ref{lem 3.1}, we know that any subsequence of $u_\varepsilon$  satisfies \eqref{eq2.21}
with $V_1=V_0=\inf_{\Omega}V$ in Lemma \ref{lem2.10}. Therefore, by Lemma \ref{lem2.10},  for any $\delta>0$,
$\lim_{\varepsilon\to0}\|u_\varepsilon\|_{L^p(\mathbb R^N\setminus (\mathcal V^\delta)_\varepsilon)}=0,$ where  $p\in (2,2^*)$.
Then the conclusion follows from  Corollary \ref{lem 2.7}.
\end{proof}
By Lemma \ref{lem 3.2}, we know that
$f_\varepsilon(x,u_\varepsilon)=u_\varepsilon$ for $\varepsilon$ small.
To drop the other penalization
terms, the key point is the following
Gaussian decay estimate for $u_\varepsilon$.
\begin{proposition}\label{pro 3.3}
For each $\delta >0$, there exist $C, c>0$ such that
$$|u_{\varepsilon}(x)|\leq  C\exp\big\{-c({\rm dist}(x, (\mathcal V^\delta)_\varepsilon))^2\big\}\quad \text{for}\ \varepsilon\in (0,\varepsilon_0)\ \text{and}\ x\in\mathbb R^N.$$
\end{proposition}
\begin{proof}
Recalling \eqref{eq2.13}, $w_\varepsilon:=|u_\varepsilon|$ satisfies
\begin{equation}\label{eq3.9}
-\Delta w_\varepsilon+w_\varepsilon+\overline V_\varepsilon (x)\widehat\eta_\varepsilon(x,w_\varepsilon)w_\varepsilon\leq w_\varepsilon\log w_\varepsilon^2,\quad x\in \mathbb R^N.
\end{equation}
By the definition of $\widetilde V_\varepsilon$, $\overline V_\varepsilon$ and $\hat\eta_\varepsilon$, for  $\varepsilon\in (0,\varepsilon_0)$,
\begin{equation*}
\overline V_\varepsilon (x)=0\quad \text{for}\ x\in B(0,{R_0}{\varepsilon^{-1}})\setminus (\mathcal V^{\delta/2})_\varepsilon.
\end{equation*}
And for $x\in \mathbb R^N\setminus B(0,{R_0}{\varepsilon^{-1}})$ if we make $\varepsilon_0$ smaller if necessary,
\begin{equation}\label{eq3.11}
\begin{aligned}
&\widehat\eta_\varepsilon(x,w_\varepsilon)=0 \quad &\text{if}\quad
w_\varepsilon(x)\geq 5\phi_\varepsilon(x),\\
&\overline V_\varepsilon (x)\widehat\eta_\varepsilon(x,w_\varepsilon)-\frac 12\log w_\varepsilon^2\geq 0\quad &\text{if}\quad
w_\varepsilon(x)\leq 5\phi_\varepsilon(x).
\end{aligned}
\end{equation}

On the other hand, by Lemma \ref{lem 3.2}, there exists $\varepsilon_\delta >0$ such that
$\|w_\varepsilon\|_{L^\infty(\mathbb R^N\setminus (\mathcal V^{\delta/2})_\varepsilon)}\leq e^{-1}$ for $\varepsilon\in(0,\varepsilon_\delta)$.
Together with \eqref{eq3.9} and \eqref{eq3.11}, we can conclude that
$w_\varepsilon$ is a weak $H^1(\mathbb R^N\setminus (\mathcal V^{\delta/2})_\varepsilon)$ solution to
\begin{equation}\label{eq3.12}
-\Delta w_\varepsilon+w_\varepsilon\leq \frac12w_\varepsilon\log w_\varepsilon^2.
\end{equation}
By compactness, there are
$k\in \mathbb N\setminus\{0\}$ depending only on $\delta>0$ and
$x_j\in \mathcal V$, $j=1,\cdots,k$ such that
$$\mathcal V^{\frac\delta2}\subset \mathcal O:= \bigcup_{j=1}^kB(x_j,\frac{2} 3\delta)\subset
\mathcal V^{\frac {2\delta}3}.
$$
We remark that
$\psi_{j,\varepsilon}(x):=\exp\big\{-\frac14\big(|x-\varepsilon^{-1}x_j|-\frac 23\delta\varepsilon^{-1}\big)^2\big\}$ satisfies
$$-\Delta \psi_{j,\varepsilon}+\psi_{j,\varepsilon}\geq \frac 12 \psi_{j,\varepsilon}\log \psi_{j,\varepsilon}^2, \quad x\in \mathbb R^N\setminus B(\varepsilon^{-1}x_j,\frac{2} 3\delta\varepsilon^{-1}).
$$
For each $x\in \mathbb R^N\setminus \mathcal O_\varepsilon$, set $\psi_\varepsilon(x)=\frac1k\sum_{j=1}^k\psi_{j,\varepsilon}(x)$.
By convexity of $s\log s^2, s\in (0,\infty)$, we have

\begin{equation}\label{eq3.13}
-\Delta \psi_\varepsilon+\psi_\varepsilon\geq \frac 12 \psi_\varepsilon\log \psi_\varepsilon^2, \quad x\in \mathbb R^N\setminus \mathcal O_\varepsilon.
\end{equation}
Since $\psi_\varepsilon(x)\geq \frac1k$ for $x\in \partial \mathcal O_\varepsilon$,  shrinking $\varepsilon_\delta>0$ if necessary and using Lemma \ref{lem 3.2},
we may assume $(w_\varepsilon-\psi_\varepsilon)^+\in H_0^1(\mathbb R^N\setminus \mathcal O_\varepsilon)$ for
$\varepsilon\in (0, \varepsilon_\delta).$ Subtracting  \eqref{eq3.13} from \eqref{eq3.12} and  testing with  $(w_\varepsilon-\psi_\varepsilon)^+$, we obtain
\begin{equation*}
\|(w_\varepsilon-\psi_\varepsilon)^+\|_{H^1(\mathbb R^N\setminus \mathcal O_\varepsilon)}^2\leq \frac12\int_{\mathbb R^N\setminus \mathcal O_\varepsilon}(w_\varepsilon-\psi_\varepsilon)^+(w_\varepsilon\log w_\varepsilon^2-\psi_\varepsilon\log \psi_\varepsilon^2)\leq 0,
\end{equation*}
where the last inequality is a result
of the decreasing monotonicity of $s\log s^2$ in $(0,e^{-1})$.
Therefore, for $\varepsilon\in (0,\varepsilon_\delta)$, $w_\varepsilon\leq \psi_\varepsilon$ in $\mathbb R^N\setminus \mathcal O_\varepsilon$.
Noting that for $x\in \mathbb R^N\setminus (\mathcal V^\delta)_\varepsilon$, $\text{dist}(x,(\mathcal V^\delta)_\varepsilon)\leq \text{dist}(x,\mathcal O_\varepsilon)$,
we have
\begin{equation}\label{3.14}
w_\varepsilon(x)\leq \psi_\varepsilon(x)\leq \exp\big\{-\frac14\big(\text{dist}(x,(\mathcal V^\delta)_\varepsilon)\big)^2\big\},\ \ \varepsilon\in (0,\varepsilon_\delta).
\end{equation}
To recover the estimate for any $\varepsilon\in (0,\varepsilon_0)$,
we note that   \eqref{eq3.12} holds for
$\varepsilon\in(0,\varepsilon_0)$ and
$x\in \mathbb R^N\setminus B(0,{R_0}{\varepsilon^{-1}})$ with $\varepsilon_0$ small but independent of $\delta$. Without loss of generality, we may also assume
$\|w_\varepsilon\|_{L^\infty(\mathbb R^N\setminus B(0,R_0\varepsilon^{-1}))}\leq e^{-1}$ for $\varepsilon\in(0,\varepsilon_0)$.
Then through a similar comparison argument, for $x\in \mathbb R^N\setminus B(0, 2R_0\varepsilon^{-1})$,
$$w_\varepsilon(x)\leq  e^{-\frac14(|x|-R_0\varepsilon^{-1})^2}\leq \exp\big\{-\frac1{16}\big(\text{dist}(x,(\mathcal V^\delta)_\varepsilon)\big)^2\big\}.
$$
For $x\in B(0, 2R_0\varepsilon^{-1})$,
we set $C_\delta:=A\exp\{\frac 1{4}R_0^2\varepsilon_\delta^{-2}\}$, where
$A:=1+\sup_{\varepsilon\in(0,\varepsilon_0)}\|u_\varepsilon\|_{L^\infty(\mathbb R^N)}$ is a positive constant by Lemma
\ref{lem 2.7}.
Then it holds
$$w_\varepsilon(x)\leq C_\delta \exp\big\{-\frac1{16}\big(\text{dist}(x,(\mathcal V^\delta)_\varepsilon)\big)^2\big\},
$$
for $\varepsilon\in[\varepsilon_\delta, \varepsilon_0)$ and $x\in \mathbb R^N$.
Recalling \eqref{3.14}, we have completed the proof.
\end{proof}
Now we are ready to show Theorem \ref{th1.1}.
\begin{proof}[Proof of Theorem \ref{th1.1}]
For $\varepsilon\in(0,\varepsilon_0)$,
by Proposition \ref{pro 3.3}, $|u_\varepsilon(x)|\leq \phi_\varepsilon(x)$ in $\mathbb R^N\setminus\Omega_\varepsilon$, which implies
$\eta_\varepsilon(x,|u_\varepsilon|)=0$, $\widehat \eta_\varepsilon(x,|u_\varepsilon|)=1$ and
 $f_\varepsilon(x,u_\varepsilon)=u_\varepsilon\log u_\varepsilon^2$. Therefore,
 $T(x,u_\varepsilon)\equiv V_\varepsilon(x)$ in \eqref{eq2.8}.
 Then
 $u_\varepsilon$ is a solution to
 \eqref{eq2}.
At this point, $v_\varepsilon(x):=u_\varepsilon(\frac{x}{\varepsilon})$, $\varepsilon\in(0,\varepsilon_0)$ is a nontrivial solution for the original equation \eqref{1.1}.
By Lemma \ref{lem 3.1}, for some $y_\varepsilon\in \mathbb R^N$ with $\text{dist}(\varepsilon y_\varepsilon, \mathcal V)\to 0$, $u_\varepsilon(\cdot+y_\varepsilon)$ converges to $ U_{V_0}$
weakly in $H^1(\mathbb R^N)$ and strongly in $L^p(\mathbb R^N)$ for $p\in (2,2^*)$.
To show it  is in fact a strong convergence in $H^1(\mathbb R^N)$, we note that  by the convergence in $L^p(\mathbb R^N)$,
$$\int_{\mathbb R^N}( u_\varepsilon^2\log  u_\varepsilon^2)^+\to\int_{\mathbb R^N}(U_{V_0}^2\log U_{V_0}^2)^+
,\ \ \varepsilon\to 0.$$
This, together with $\Gamma_\varepsilon'(u_\varepsilon)u_\varepsilon=0$ and $I_{V_0}'(U_{V_0})U_{V_0}=0$,
leads us to the fact
$$\lim_{\varepsilon\to 0}\int_{\mathbb R^N}|\nabla  u_\varepsilon|^2
+ \widetilde V_\varepsilon(x) u_\varepsilon^2
+( u_\varepsilon^2\log  u_\varepsilon^2)^-
=\int_{\mathbb R^N}|\nabla U_{V_0}|^2+V_0 U_{V_0}^2+(U_{V_0}^2\log U_{V_0}^2)^-.$$
Therefore, $\lim_{\varepsilon\to 0}|\nabla  u_\varepsilon|_2^2= |\nabla U_{V_0}|_2^2$ and $\lim_{\varepsilon\to 0}|u_\varepsilon|^2_2=|U_{V_0}|^2_2$ by $\widetilde V \geq V_0$ in $\Omega$ and Proposition \ref{pro 3.3}. Thus we have proved the convergence in $H^1(\mathbb R^N)$.
  Moreover, we may assume without loss of generality that $u_\varepsilon$ attains its maximum value at $y_\varepsilon$   which is also the  unique local maximum point by the same arguments in \cite[Proposition 2.1]{DelPino-Felmer1996}.
  This completes the proof of (i).
  The conclusion (ii) follows from (i) and an argument similar   to Proposition
  \ref{pro 3.3}.
\end{proof}

\section{Proof of Theorem \ref{th1.2}}
In this section, we show Theorem \ref{th1.2}
and assume (V0), (V2) and (V3).
By mini-max theory, these conditions guarantee
the existence of a critical point of $V$ at level $\mu_0$ inside $\Omega$.
Similar to Section \ref{sec3},
we assume $0\in\set{x\in\Omega|V(x)=\mu_0}\subset \Omega\subset B(0,R_0/2)$ for some $R_0>0$ and $V\geq 1$ in $B(0,R_0)$.

Let $\mu_1\in(\mu_0-\log2,\mu_0)$ be a fixed number which is so close to $\mu_0$ that
\begin{equation}\label{eq41}
\partial_TV(x)\neq 0, \quad \text{for}\
x\in \partial\Omega \cap\set{x|\mu_1<V(x)\leq \mu_0}.
\end{equation}
We resize
$$\widetilde\Omega=\Omega\cap\set{x\in\Omega| V(x)>\mu_1}.$$
And choose
$$\mu_{1,\varepsilon}:=\inf\Set{\zeta> \mu_1|\text{dist}(\set{x\in\Omega| V(x)=\zeta},\partial\widetilde\Omega)\geq \varepsilon^{\frac12}}.
$$
Then $\mu_{1,\varepsilon}\to \mu_1$ as $\varepsilon\to0$. For each small $\varepsilon>0$,
we choose $\gamma_\varepsilon\in \mathcal T$ such that
$$\max_{y\in \mathcal  B}V(\gamma_\varepsilon(y))<\mu_0+\varepsilon.$$
Denote
\begin{align*}
B=\gamma_\varepsilon(\mathcal B)\cap \set{x\in\Omega | V(x)\geq \mu_{1,\varepsilon}},\quad
S=\gamma_\varepsilon(\mathcal B)\cap \set{x\in\Omega | V(x)= \mu_{1,\varepsilon}}
\end{align*}
and note that $\mathcal B_0\neq \emptyset$ implies $S\neq \emptyset$ since $\gamma_\varepsilon(\mathcal B)$ is connected and $\gamma_\varepsilon(y)=y$ for $y\in\mathcal B_0$.
Remark also that $\text{dist}(B, \partial \widetilde\Omega)=\text{dist}(S, \partial \widetilde\Omega)=\varepsilon^{\frac12}$ by the choice of $\mu_{1,\varepsilon}$.
Set $$\widetilde{\mathcal T}=\set{\gamma\in C(B,\widetilde\Omega) | \gamma(x)=x\ \text{for each}\ x\in S}.$$
Then for each $\gamma\in \widetilde{\mathcal T}$, we can find $\tau\in \mathcal T$ as
$$\tau(y)=\begin{cases}
\gamma_\varepsilon(y) &\text{if}\ \gamma_\varepsilon(y)\notin B,\cr
\gamma(\gamma_\varepsilon(y)) &\text{if}\ \gamma_\varepsilon(y)\in B,
\end{cases}
$$
satisfying $\sup_{x\in B}V(\gamma(x))=\sup_{y\in \mathcal B}V(\tau(y))\in[\mu_0,+\infty)$.
Therefore,
$$\mu_{0,\varepsilon}:=\inf_{\gamma\in \widetilde{\mathcal T}}\max_{x\in B}V(\gamma(x))\in [\mu_0,\mu_0+\varepsilon).
$$
Without loss of generality, we may assume that $B$ is  also connected, since there exists a connected component of $B$ such that the restriction of each $\gamma\in\widetilde{\mathcal T}$ to this component across the level set $\set{x|V(x)\geq \mu_{0,\varepsilon}}$.
\begin{remark}\label{rek4.1}
Since the choice of $\widetilde \Omega$
is independent of $\varepsilon$,
the arguments in Section \ref{sec2} are applicable to it.
By \eqref{eq41}, there holds
\begin{equation*}
\partial_TV(x)\neq 0, \quad \text{for}\
x\in \partial\widetilde\Omega \cap\set{x|\mu_1<V(x)\leq \mu_0}.
\end{equation*}
Thus in what follows in this section, we can denote  $\widetilde \Omega$ as $\Omega$ for the sake of brevity.
We also assume $\varepsilon_0>0$ determined in Section \ref{sec2} is such that $\varepsilon_0<\mu_0-\mu_1$.
\end{remark}

By  Remark \ref{rek2.9} (ii), it is  nature to  consider the   Nehari manifold:
$$\mathcal M_\varepsilon:=\set{ u\in H_\varepsilon | \Gamma_\varepsilon'(u)u=0\ \text{and}\ \|u\|_\varepsilon\geq \varepsilon^2}.$$
\begin{lemma}\label{lem4.1}
$\mathcal M_\varepsilon$ is a $C^1$ manifold with co-dimensional $1$ and
 $\inf_{u\in\mathcal M_\varepsilon}\Gamma_\varepsilon(u)$ is attained by a critical point of $\Gamma_{\varepsilon}$. Moreover,
 \begin{equation}\label{42}
 \liminf_{\varepsilon\to0}\inf_{u\in\mathcal M_\varepsilon}\Gamma_\varepsilon(u)>0.
 \end{equation}
\end{lemma}
\begin{proof}
By  Remark \ref{rek2.9} (ii), we may assume without loss of generality that for some constant $\sigma>0$,
$$ \inf_{u\in \mathcal M_\varepsilon}\|u\|_\varepsilon\geq \sigma.
$$
For $u\in\mathcal M_\varepsilon$, since $\Gamma_\varepsilon'(u)u=0$, by Lemma \ref{lem etaq} (iii) and Corollary \ref{cor2.3}, we have for some $C>0$ independent of $\varepsilon$ and
$u$,
$$\begin{aligned}
\frac{\|u\|_\varepsilon^2}2\leq \int_{\mathbb R^N}(1-\chi_\varepsilon)(u^2\log u^2)^+&\leq C\int_{\mathbb R^N}(1-\chi_\varepsilon)|u|^{2+\frac{2}{N+2}}\\
&\leq C\|u\|_\varepsilon^{1+\frac{2}{N+2}}\Big(\int_{\mathbb R^N}(1-\chi_\varepsilon)|u|^2\Big)^\frac12.
\end{aligned}
$$
Therefore for some $\sigma_0>0$ independent of $\varepsilon$  and $u$,
$$\int_{\mathbb R^N}(1-\chi_\varepsilon)|u|^2\geq \sigma_0.
$$
Hence, by $g(s)s-2G(s)\geq 0$,
$$\begin{aligned}
2\Gamma_\varepsilon(u)=
2\Gamma_\varepsilon(u)-\Gamma_\varepsilon'(u)u
=&\int_{\mathbb R^N}(1-\chi_\varepsilon)u^2+\int_{\mathbb R^N}\chi_\varepsilon\big(g(u)u-2G(u)\big)\geq \sigma_0.
\end{aligned}
$$
Then there holds \eqref{42}.
To show $\mathcal M_\varepsilon$ is a $C^1$ manifold with co-dimensional $1$, for $u\in H_\varepsilon$,
set
\begin{equation}\label{F}
\mathcal F_\varepsilon(u)=\Gamma_\varepsilon'(u)u=\|u\|^2_\varepsilon+\Psi'(u)u-\int_{\mathbb R^N}f_\varepsilon(x,u)u.
\end{equation}
Noting that
$f_\varepsilon(x,s)s=2F_\varepsilon(x,s)+(1-\chi_\varepsilon(x))s^2+\chi_\varepsilon(x)(g(s)s-2G(s))$,
we can get  $\mathcal F_\varepsilon\in C^1(H_\varepsilon)$ by Lemma \ref{lem etaq} (v).
Then it is direct to check that for all small $\varepsilon$ and  $u\in\mathcal M_\varepsilon$, by Corollary \ref{cor2.3} and  $(g(s)s-2G(s))'s\geq g(s)s-2G(s)\geq0$,
\begin{equation}\label{eq2.17}
\begin{aligned}
\mathcal F_\varepsilon'(u)u&\leq -2\int_{\mathbb R^N} (1-\chi_\varepsilon)u^2-\int_{\mathbb R^N}\chi_\varepsilon \big(g(s)s-2G(s)\big)+Ce^{-c/\varepsilon}\\
&\leq  -2\sigma_0+Ce^{-c/\varepsilon}<-\sigma_0.
\end{aligned}
\end{equation}
Therefore, $\mathcal M_\varepsilon$ is a $C^1$ manifold with co-dimensional $1$.

To complete the proof, it is easy to check,  by Ekeland variational principle and Corollary \ref{cor 2.4} that
$\inf_{u\in\mathcal M_\varepsilon}\Gamma_\varepsilon(u)$ is achieved by a critical point of $\Gamma_\varepsilon$.
\end{proof}

We remark that the infimum defined  in the above proposition
 does not necessarily determine a solution to the original problem under the assumptions in this section.
We should define  another minimax  value on the Nehari manifold to
solve the equation. To this end, for $\varepsilon\in (0,\varepsilon_0)$,
let $\zeta_\varepsilon(x)=\zeta(\varepsilon^{1/3} x)$, where $\zeta$ is a  radial smooth cut-off such that
 $\zeta (x)=1$ if $|x|\leq 1/2$,
 $\zeta (x)=0$  if $|x|\geq 1$ and $|\nabla\zeta|\leq 4$.
For  any $y\in  B$, define
\begin{equation}\label{3.10}
 \psi_\varepsilon(y)(\cdot):=\zeta_\varepsilon(\cdot-\varepsilon^{-1}y) U_{V(y)}(\cdot-\varepsilon^{-1} y).
\end{equation}
By Remark \ref{rek4.1} and the construction of $B$ and $S$, when $\varepsilon$ is small and $y\in B$, we have
\begin{equation} \label{3.11}
\Psi_\varepsilon(s\psi_\varepsilon(y))=0 \ \ \ \text{and}\ \
 f_\varepsilon(x,s\psi_\varepsilon(y))= s \psi_\varepsilon(y) \log (s\psi_\varepsilon(y))^2,\ \ \ \text{for\ all } s\in\mathbb R.
\end{equation}
Note that by the monotonicity of $\log s$, $s>0$,
there is a unique $t_\varepsilon(y)>0$ such that $t_\varepsilon(y)\psi_\varepsilon(y)\in\mathcal M_\varepsilon$.
By uniqueness,  $t_\varepsilon(y)$ is continuous for $y\in B$.
We  define the min-max value
 \begin{equation*}
 m_\varepsilon=\inf_{\phi\in\Phi_\varepsilon}\max_{y\in B}\Gamma_\varepsilon(\psi(y)),
 \end{equation*}
with
 \begin{equation*}
 \Phi_\varepsilon=\set {\phi\in C(B, \mathcal M_\varepsilon)| \phi(y)=t_\varepsilon(y)\psi_\varepsilon(y)\ \text{for}\ y\in S}.
 \end{equation*}
Then by the continuity of $\Gamma_\varepsilon(t\psi_\varepsilon(y))$ and compactness of $B\subset \Omega$, we can show
\begin{lemma}\label{lemma4.1} There is $T_0>0$ independent of $\varepsilon$ and $y$ such that $t_\varepsilon(y)\in(0, T_0)$. Moreover
\begin{equation*}
\begin{aligned}
\limsup_{\varepsilon\to 0}\sup_{y\in B}\Gamma_\varepsilon(t_\varepsilon(y)\psi_\varepsilon(y))\leq  m(\mu_0),\quad
\limsup_{\varepsilon\to 0}\sup_{y\in S}\Gamma_\varepsilon(t_\varepsilon(y)\psi_\varepsilon(y))\leq m(\mu_1).
\end{aligned}
\end{equation*}
\end{lemma}
 \begin{proof}
As $\varepsilon\to 0$, there uniformly hold for $y\in  B$ that
 \begin{align*}
 \|\psi_\varepsilon(y)\|^2_\varepsilon&\to \int_{\mathbb R^N}|\nabla U_{V(y)}|^2+V(y)|U_{V(y)}|^2\quad\text{and}\\
 \int_{\mathbb R^N} \psi_\varepsilon(y)^2\log \psi_\varepsilon(y)^2&\to \int_{\mathbb R^N}U_{V(y)}^2\log U_{V(y)}^2=\int_{\mathbb R^N}|\nabla U_{V(y)}|^2+V(y)|U_{V(y)}|^2.
 \end{align*}
 By \eqref{3.11}, it is easy to check that for all $y\in B$, there uniformly holds
 $$\begin{aligned}
 \Gamma_\varepsilon(t\psi_\varepsilon(y))&=
 \frac{1}2t^2(1-\log t^2)\int_{\mathbb R^N}|U_{V(y)}|^2+t^2(1+|\log t^2|)o_\varepsilon(1)\\
&=\frac12 t^2(1-\log t^2)e^{V(y)}|U|_2^2+t^2(1+|\log t^2|)o_\varepsilon(1).
 \end{aligned}
 $$
Then for small $\varepsilon$ and $y\in B$, we have $\Gamma_\varepsilon(t\psi_\varepsilon(y))<-2$ if $t\geq T_0$ with $T_0$ a fixed large constant. So for $y\in  B$,
 we can  check that $\sup_{t\in[0,+\infty)}\Gamma_\varepsilon(t\psi_\varepsilon(y))
 =\sup_{t\in[0,T_0]}\Gamma_\varepsilon(t\psi_\varepsilon(y))\leq \frac12e^{V(y)}|U|_2^2+o_\varepsilon(1)=m(V(y))+o_\varepsilon(1)$. And the conclusion follows from the choice of $B$ and $S$.
 \end{proof}

Next we have the following lemma.
\begin{lemma}\label{lemma4.2}
$m(\mu_1)<\liminf_{\varepsilon\to 0}m_\varepsilon\leq \limsup_{\varepsilon\to 0}m_\varepsilon \leq m(\mu_0)$.
\end{lemma}
\begin{proof}

By  Lemma \ref{lem2.10}, Lemma \ref{lem4.1} and Lemma \ref{lemma4.1}, it is clear that
$$ m(\mu_1)\leq \liminf_{\varepsilon\to0}\inf_{u\in\mathcal M_\varepsilon}\Gamma_\varepsilon(u)\leq\liminf_{\varepsilon\to0}m_\varepsilon\leq \limsup_{\varepsilon\to0} m_\varepsilon\leq m(\mu_0).$$
To see $\liminf_{\varepsilon\to0}m_\varepsilon>m(\mu_1)$, we argue by contradiction that there is $\varepsilon_n\to 0$ such that
$\lim_{n\to\infty}m_n=m(\mu_1)$ with $m_n:=m_{\varepsilon_n}$.
Let us take $\phi_n\in \Phi_{\varepsilon_n}$
 such that
\begin{equation}\label{eq4.5}
\lim_{n\to\infty}\sup_{y\in B}\Gamma_{\varepsilon_n}(\phi_n(y))= m(\mu_1)\leq \liminf_{n\to\infty}\inf_{u\in\mathcal M_\varepsilon}\Gamma_{\varepsilon_n}(u).
\end{equation}
Since $\partial \Omega$ is $C^1$ and compact, we can choose small $\delta_0>0$ (will be fixed later)  such that the projection
$$\pi_n:\Omega^{\delta_0} \to \set{x\in \Omega| \mbox{dist}(x,\partial\Omega)\geq\varepsilon_n^\frac12}$$
which maps a point in $\Omega^{\delta_0}$ to its unique closest point in $\set{x\in \Omega| \mbox{dist}(x,\partial\Omega)\geq\varepsilon_n^\frac12}$ is continuous. In particular,
\begin{equation}\label{eq4.6}
\pi_n=id\quad \text{on}\quad S.
\end{equation}
For each $u\in L^p(\mathbb R^N)\setminus\{0\}$ with $p\in (2,2^*)$, define the barycenter type function
$$\beta_n(u)=\frac{\int_{(\Omega^{1})_{\varepsilon_n}}\varepsilon_nx |u(x)|^pdx}{\int_{\mathbb R^N} |u(x)|^pdx}.$$

We claim that
\begin{equation}\label{eq4.7}
\lim_{n\to0}\sup_{y\in B}\text{dist}(\beta_n(\phi_n(y)), A_1)=0,\quad\mbox{where}\ A_1=\set{y\in\overline\Omega|V(y)=\mu_1}.
\end{equation}
In fact, let $v_n\in \phi_n(B)$  be
such that $\lim_{n\to\infty}\text{dist}(\beta_n(v_n),A_1)>0$. Then
by \eqref{eq4.5} and the Ekeland Variational Principle, there is $u_n\in \mathcal M_{\varepsilon_n}$ such that
\begin{align*}
&\lim_{n\to\infty}\|u_n-v_n\|_{\varepsilon_n}=0,\quad\Gamma_{\varepsilon_n}(u_{n})\leq \Gamma_{\varepsilon_n}(v_{n})\quad \text{and}\\
&\lim_{n\to\infty}\|\Gamma_{\varepsilon_n}'(u_{n})-\lambda_{n}\mathcal F_{\varepsilon_n}'(u_{n})\|_{H^{-1}_{\varepsilon_n}}= 0,
\end{align*}
where $\mathcal F_{\varepsilon_n}$ is defined by \eqref{F} for $\varepsilon=\varepsilon_n$ and   $\lambda_n\in\mathbb R$ is such that
$$\|\Gamma_{\varepsilon_n}'(u_{n})-\lambda_{n}\mathcal F_{\varepsilon_n}'(u_{n})\|_{H^{-1}_{\varepsilon_n}}=\min_{\lambda\in \mathbb R}\|\Gamma_{\varepsilon_n}'(u_{n})-\lambda\mathcal F_{\varepsilon_n}'(u_{n})\|_{H^{-1}_{\varepsilon_n}}.$$
Since $\lim_{n\to\infty}\Gamma_{\varepsilon_n}(u_n)\leq m(\mu_1)$, by Lemma \ref{lem 2.3}, $\|u_n\|_{\varepsilon_n}$ and $\int_{\mathbb R^N} u_n^2|\log u^2_n|$ are bounded.
Then by \eqref{eq2.17}, we have $0=\lim_{n\to\infty}|\Gamma_{\varepsilon_n}'(u_n)u_n-\lambda_n\mathcal F_{\varepsilon_n}'(u_n)u_n|\geq \sigma_0\lim_{n\to\infty}|\lambda_n|$. Thus,  $\lim_{n\to\infty}\lambda_n=0$,
which implies, by \cite[Lemma 2.1(iii)]{Tanaka-Zhang}, for any $w\in \mathcal D$,
 $$\lim_{n\to\infty}\sup_{y\in \mathbb R^N} |\Gamma_{\varepsilon_n}'(u_n)w(\cdot-y)|=\lim_{n\to\infty}\sup_{y\in\mathbb R^N} |\Gamma_{\varepsilon_n}'(u_n)w(\cdot-y)-\lambda_n\mathcal F_{\varepsilon_n}'(u_n)w(\cdot-y)|=0.
 $$
 Then checking the proof directly, it is clear that the same conclusion in Lemma \ref{lem2.10} holds for $u_n$. Thus up to a subsequence, we can find $z_n$ with $\varepsilon_nz_n\to z_0\in A_1$ such that
$u_n(\cdot+z_n)\to U_{\mu_1}$ strongly in $L^p(\mathbb R^N)$. Since $U_{\mu_1}$ is radially symmetric, we can easily
check that $\beta_n(v_n)\to z_0$, which is a contradiction.

Note that
$\pi_n\circ\beta_n\circ\phi_n$ maps continuously from $B$ to $\Omega$
 and that $\pi_n\circ\beta_n\circ\phi_n$ is
identity on $S$ by \eqref{eq4.6}. Therefore,
$\pi_n\circ\beta_n\circ\phi_n\in \widetilde {\mathcal T}$ and there is $y_n\in B$ such that
$$\mu_0\leq \mu_{0,\varepsilon_n}\leq
V(\pi_n\circ\beta_n\circ\phi_n(y_n) ).
$$
Moreover, according to the definition of $\pi_n$, we have
$|\pi_n\circ\beta_n\circ\phi_n(y_n)- \beta_n\circ\phi_n(y_n) |\leq 2\delta_0$.
By this and the fact that $|\nabla V(x)|$  is bounded in $\Omega^{2\delta_0}$, if we make $\delta_0$ smaller,  we shall have
$V(\beta_n\circ\phi_n(y_n))\geq \frac{1}{2}(\mu_0+\mu_1)$, which
 contradicts to \eqref{eq4.7}.
\end{proof}
Now we are ready to proof Theorem \ref{th1.2}.
\begin{proof}[Proof of theorem \ref{th1.2}]
According to Lemma \ref{lemma4.1}, Lemma  \ref{lemma4.2} and Corollary \ref{cor 2.4}, for each $\varepsilon\in(0,\varepsilon_0)$, we apply the minimax principle
(see e.g. \cite{willem}) to obtain a critical point $u_\varepsilon$ of $\Gamma_\varepsilon$  such that
\begin{equation}\label{4.10}
\Gamma_\varepsilon(u_\varepsilon)=m_\varepsilon\ \ \ \text{and}\ \ \
m(\mu_1)<\liminf_{\varepsilon\to 0}m_\varepsilon\leq \limsup_{\varepsilon\to 0}m_\varepsilon\leq   m(\mu_0).
\end{equation}
By similar arguments in Proposition \ref{pro 2.6},
we may assume that $u_\varepsilon>0$.
Note that by our choice of $\mu_1$, $m(\mu_0)<2m(\mu_1)$.
Then by Lemma \ref{lem2.10}, up to a subsequence,  we can obtain  that
$\widetilde u_\varepsilon:=u_\varepsilon(\cdot+x_\varepsilon)\to u$ weakly in $H^1(\mathbb R^N)$ and strongly in $L^p(\mathbb R^N)$ for $p\in (2,2^*)$, where  $x_\varepsilon\in \mathbb R^N$ satisfies
 $\varepsilon x_\varepsilon\to x_0\in \overline\Omega$ and $u=U_{V(x_0)}\in \mathcal D$  is the positive radial solution to
$-\Delta u +V(x_0) u=u\log u^2$.
 Moreover,  similarly to Proposition \ref{pro 3.3}, there are $C, c>0$ such that
\begin{equation}\label{4.11}
|u_\varepsilon(x)|\leq C e^{-c|x-x_\varepsilon|^2}.
 \end{equation}
 This in turn implies that $u_\varepsilon$ solves
 \begin{equation}\label{4.12}
 -\Delta u_\varepsilon +V_\varepsilon(x) u_\varepsilon=f_\varepsilon(x,u_\varepsilon).
 \end{equation}
 Using the $L^p$-estimates (\cite{Trudinger1997}) and \eqref{4.11}, we can deduce
 \begin{equation}\label{eq4.12}
|\nabla u_\varepsilon(x)|+|u_\varepsilon(x)|\leq C e^{-c|x-x_\varepsilon|^2}.
 \end{equation}
 By the convergence of $\widetilde u_\varepsilon$ in $L^p(\mathbb R^N)$ and
  $|(s^2\log s^2-g(s)s)'|\leq C|s|^{p-1}$, we have
$$\int_{\mathbb R^N}(1-\chi_\varepsilon)(u_\varepsilon^2\log u_\varepsilon^2-g(u_\varepsilon)u_\varepsilon)\to\int_{\mathbb R^N}(1-\widetilde \chi)(u^2\log u^2-g(u)u),
$$
where $\widetilde \chi$ is the limit function of $\chi_\varepsilon(\cdot+x_\varepsilon)$ as in \eqref{eq2.23} satisfying
$(1-\widetilde \chi)u\log u^2+\widetilde\chi g(u)=u\log u^2$. Especially, we recall that in the proof of Lemma \ref{lem2.10}, if $\limsup_{\varepsilon\to0}{\rm dist}(x_\varepsilon,\partial \Omega_\varepsilon)<\infty$, then $\widetilde\chi$ is the characteristic function of a half space $H$ and
\begin{equation}\label{4.13}
|u(x)|\leq e^{-1},\ \ x\in H.
\end{equation}
Then  we deduce
by \eqref{eq2.23}, \eqref{4.12} and the Fatou's Lemma that
$$\begin{aligned}
\limsup_{\varepsilon\to 0}\int_{\mathbb R^N}|\nabla u_\varepsilon|^2
+ V_\varepsilon u_\varepsilon^2
&=\int_{\mathbb R^N}\Big(|\nabla u|^2+V(x_0) u^2-g(u)u\Big)+\liminf_{\varepsilon\to0}\int_{\mathbb R^N}g(u_\varepsilon)u_\varepsilon\\
&\leq \int_{\mathbb R^N}\Big(|\nabla u|^2+V(x_0) u^2\Big).
\end{aligned}$$
By $\varepsilon x_\varepsilon\to x_0\in \overline \Omega$ and \eqref{eq4.12}, we can conclude that
$\widetilde u_\varepsilon\to u$ strongly in
$H^1(\mathbb R^N)$.
And as a result of \eqref{4.10}, we have $V(x_0)\in (\mu_1,\mu_0]$.

Next we show that $x_0\in \Omega$.
If $x_0\in\partial \Omega$,
 then in a small neighborhood $B(x_0,\rho)$,  $ \Omega$ can be described by
 $$\Omega\cap B(x_0,\rho)=\set{x\in B(x_0,\rho)|   h (x)< 0},$$
 where  $h$ is a smooth function such that $h(x)=0$ and $|\nabla h(x)|=1$ for $x\in\partial \Omega$. Then the unit outward normal vector to $\partial\Omega$ at $x$ is $\nabla h(x)$.
On the other hand, by $V(x_0)\in (\mu_1,\mu_0]$ and Remark \ref{rek4.1},
we can fix $\rho$ small enough such that $\frac{\partial V(x)}{\partial\nu_x}\geq 2\tau>0$ for all $x\in B(x_0, 2\rho)$, where $\nu_x$ denotes the unit vector in the direction of  $\nabla V(x)$ projected
to the tangential space of $\partial\Omega$ at $y$
with $\mbox{dist}(x,y)=\mbox{dist}(x,\partial\Omega)$,
which implies
\begin{equation}\label{4.14}
\liminf_{\varepsilon\to0}\inf_{x\in \Lambda_\varepsilon} \frac{\partial V(\varepsilon x)}{\partial\nu_{\varepsilon y_\varepsilon}} \geq \tau >0,\quad\mbox{where}\ \Lambda_\varepsilon=B(x_\varepsilon, \rho/\varepsilon),
\end{equation}
and $y_\varepsilon\in \partial \Omega_\varepsilon$ be such that $\mbox{dist}(x_\varepsilon, y_\varepsilon)=\mbox{dist}(x_\varepsilon, \partial \Omega_\varepsilon)$.
Setting $\nu_\varepsilon:=\nu_{\varepsilon y_\varepsilon}$,
multiplying \eqref{4.12} by $\nabla u_\varepsilon\cdot \nu_\varepsilon$ and integrating by parts in $\Lambda_\varepsilon=B(x_\varepsilon, \rho/\varepsilon)$, we have
\begin{equation}\label{4.16}
\begin{aligned}
&\varepsilon\Big(\int_{\partial \Lambda_\varepsilon}\Big(|\nabla u_\varepsilon|^2+V_\varepsilon u_\varepsilon^2\Big) \frac{x-x_\varepsilon}{\rho}\cdot \nu_\varepsilon
-\int_{\Lambda_\varepsilon} \frac{\partial V(\varepsilon x)}{\partial \nu_\varepsilon}u_\varepsilon^2\Big)\\
=&2\int_{\Omega_\varepsilon\cap\Lambda_\varepsilon}\frac{\partial u_\varepsilon}{\partial \nu_\varepsilon}u_\varepsilon\log u_\varepsilon^2+2\int_{\Omega_\varepsilon\setminus \Lambda_\varepsilon}\frac{\partial u_\varepsilon}{\partial \nu_\varepsilon}g(u_\varepsilon)\\
=&\varepsilon\int_{\partial \Lambda_\varepsilon}F_\varepsilon(x,u_\varepsilon)\frac{x-x_\varepsilon}{\rho}\cdot \nu_\varepsilon
+\int_{\partial \Omega_\varepsilon\cap \Lambda_\varepsilon}\Big(u_\varepsilon^2\log u_\varepsilon^2-u_\varepsilon^2-2G(u_\varepsilon)\Big)\nabla h(\varepsilon x)\cdot \nu_\varepsilon.
\end{aligned}
\end{equation}
If $\limsup_{\varepsilon\to0}\mbox{dist}(x_\varepsilon,\partial \Omega_\varepsilon)=\infty$,
then by \eqref{eq4.12}, up to a subsequence, we  assume without loss of generality that
$u_\varepsilon^2\log u_\varepsilon^2-u_\varepsilon^2-2G(u_\varepsilon)=0$ on $\partial\Omega_\varepsilon$.
As a result, by \eqref{eq4.12}, \eqref{4.14} and \eqref{4.16},
we obtain
$$\tau\int_{\mathbb R^N} u^2\leq \liminf_{\varepsilon\to 0}\int_{\Lambda_\varepsilon} \frac{\partial V(\varepsilon x)}{\partial \nu_\varepsilon}u_\varepsilon^2\leq 0,
$$
which is a contradiction.
If $\limsup_{\varepsilon\to0}\mbox{dist}(x_\varepsilon,\partial \Omega_\varepsilon)<\infty$,
again by \eqref{eq4.12}, there is $r>0$ sufficiently large such that $u_\varepsilon^2\log u_\varepsilon^2-u_\varepsilon^2-2G(u_\varepsilon)=0$ for $x\in (\partial \Omega_\varepsilon\cap \Lambda_\varepsilon)\setminus B(x_\varepsilon,r)$. For
$x\in \partial \Omega_\varepsilon\cap B(x_\varepsilon,r)$,
 we calculate that
\begin{equation*}
\begin{aligned}
 \int_{\partial \Omega_\varepsilon\cap B(x_\varepsilon,r)}\Big(u_\varepsilon^2&\log u_\varepsilon^2-u_\varepsilon^2-2G(u_\varepsilon)\Big)\nabla h(\varepsilon x)\cdot \nu_\varepsilon\\
 =&\int_{(\partial \Omega_\varepsilon-x_\varepsilon)\cap B(0,r)}\Big(\widetilde u_\varepsilon^2\log \widetilde u_\varepsilon^2-\widetilde u_\varepsilon^2-2G(\widetilde u_\varepsilon)\Big)\nabla h(\varepsilon x+\varepsilon x_\varepsilon)\cdot \nu_\varepsilon,
 \end{aligned}\end{equation*}
 where $\partial \Omega_\varepsilon-x_\varepsilon=\set{x-x_\varepsilon|x\in\partial \Omega_\varepsilon}$.
For $x\in (\partial \Omega_\varepsilon-x_\varepsilon)\cap B(0, r)$, there uniformly holds
 $$|\varepsilon^{-1}\nabla h(\varepsilon x+\varepsilon x_\varepsilon)\cdot \nu_\varepsilon|= |\varepsilon^{-1}(\nabla h(\varepsilon x+\varepsilon x_\varepsilon)-\nabla h(\varepsilon y_\varepsilon))\cdot \nu_\varepsilon|\leq Cr.
 $$
 Therefore,
dividing \eqref{4.16} by $\varepsilon$
and letting $\varepsilon\to 0$, by \eqref{eq4.12}--\eqref{4.14}, we obtain
$$\tau\int_{\mathbb R^N} u^2\leq Cr
\int_{\partial H}\Big(u^2\log u^2-u^2-2G(u)\Big)=0,
$$
where we also used the uniform convergence  $u_\varepsilon(x+x_\varepsilon)\to u(x)$ in $ B(0,r)$.
Then we get a contradiction and have shown that $x_0\in \Omega$. Therefore,
by \eqref{4.11} and $\varepsilon x_\varepsilon\to x_0$, $f_\varepsilon(x,u_\varepsilon)=u_\varepsilon\log u_\varepsilon^2$ and $u_\varepsilon$ is a solution to \eqref{eq2}. In particular,
we may assume $x_\varepsilon$ is the unique local maximum point of $u_\varepsilon$. Finally an  argument similar to \cite{DelPino-Felmer1997}  gives a way to find a family of solutions $\{u_\varepsilon\}$ with maximum points $\{x_\varepsilon\}$ such that  $V(\varepsilon x_\varepsilon)\to \mu_0$ and $\nabla V(\varepsilon x_\varepsilon)\to0$.
\end{proof}


\section{Proof of Theorem \ref{th2}}
In this section, we assume (P1) and (P2) hold. Note that
$\liminf_{|x|\to\infty}V(x)|x|^{-\mu}>-\infty$ implies  $\liminf_{|x|\to\infty}V(x)|x|^{-\mu'}>-\infty$ for $\mu'\leq \mu$.
So, without loss of generality, we may assume $\mu=2-2\kappa$ and $\kappa\in [0,1]$ in (P2).
By (P1),
$$\mathcal P=\{ x \in \Omega:\   P(x) =P_0\}$$ is a nonempty compact subset of $\Omega$.
Without loss of generality, assume $0\in\mathcal P\subset\Omega\subset B(0,R_0/2)$, where $R_0>0$  is a constant fixed in order that
$K(x)\geq \sigma_1 |x|^{-\kappa}$ in $\mathbb R^N\setminus B(0,R_0)$ and $K(x)\geq \sigma_2>0$  in $B(0,R_0)$
for some constant $\sigma_1\in (0,\liminf_{x\to\infty}K(x)|x|^{\kappa})$ and $\sigma_2>0$.

Moreover, we can also assume without loss of generality that
\begin{equation*}
V(x)\geq 1,\ \ \text { for\ }x\in B(0,R_0).
\end{equation*}
By setting $u(x)=v(\varepsilon x)$, we consider the equation
\begin{equation}\label{eq5.1}
-\Delta u+V(\varepsilon x)u=K(\varepsilon x)u\log u^2 \  \ \text { in}\ \ \mathbb R^N.
\end{equation}
We first renew some notation in Section \ref{sec2}.
Redefine
$\eta_\varepsilon(x,t)$ and $\widehat\eta_\varepsilon(x,t)$ in \eqref{eta} and \eqref{hateta} with
\[\phi_\varepsilon(x)=\exp\big\{-\varepsilon^{1-\kappa}|x|^{2-\kappa}\big\}.
\]
Reset
$\widetilde V$ in \eqref{tildev} as
\begin{equation*}
\widetilde V(x)=
\begin{cases}
V(x),\ &|x|\leq R_0;\\
\max\{V(x),|x|^{2-2\kappa}\}, \ &|x|\geq R_0,
\end{cases}
\end{equation*}
and renew $\widetilde V_\varepsilon$, $\overline V_\varepsilon$, $\Psi_\varepsilon$ and $H_\varepsilon$ correspondingly.
Redefine
\begin{equation*}
\Gamma_{\varepsilon}(u)=\frac{1}{2}\|u\|_\varepsilon^2+\Psi_\varepsilon(u)
-\frac{1}{2}\int_{\mathbb R^N}K(x)F_\varepsilon(x,u)dx,\quad u\in H_\varepsilon,
\end{equation*}
for the newly defined $\widetilde V_\varepsilon$, $\overline V_\varepsilon$, $\Psi_\varepsilon$ and $H_\varepsilon$.
We note that in the case $\kappa =1$, the redefined $\Gamma_\varepsilon$ is not necessarily well-defined  on the new $H_\varepsilon$.
To overcome this difficulty, as in \cite{ITWZ}, we introduce for $R\geq R_0$ the Hilbert space
$$H_{\varepsilon,R}=\Set{u\in H_\varepsilon |
\int_{\mathbb R^N}|x|^2u^2<\infty }
$$
with inner product and norm
$$(u,v)_{\varepsilon,R}=(u,v)_\varepsilon+\int_{\mathbb R^N}\Big[\big(|x|-\frac R\varepsilon\big)^+\Big]^2uv,\quad
\|u\|_{\varepsilon,R}=(u,u)_{\varepsilon,R}^\frac12.$$
For each $R\geq R_0$, we consider  the functional
\begin{equation*}
\Gamma_{\varepsilon,R}(u)=\frac{1}{2}\|u\|_{\varepsilon,R}^2+\Psi_\varepsilon(u)
-\frac{1}{2}\int_{\mathbb R^N}K(\varepsilon x)F_\varepsilon(x,u)dx,\quad u\in H_{\varepsilon,R},
\end{equation*}
which is $C^1$ on $H_{\varepsilon,R}$ by Lemma \ref{lem 2.1}.
 We can check that the conclusions of Lemma \ref{lem etaq}, Corollary \ref{cor2.3}, Lemma \ref{lem 2.3}, Corollary \ref{cor 2.4} and Lemma \ref{lem 2.5} still hold for $\Gamma_{\varepsilon, R}$ by similar arguments.
 Therefore, we can get the following existence result.
 \begin{proposition}\label{pro4.1} There exist $\varepsilon_0>0$ and $M_1>M_0>0$ such that for each $\varepsilon\in (0, \varepsilon_0)$ and $R\geq R_0$, $\Gamma_{\varepsilon, R}$ admits a nontrivial critical point $u_{\varepsilon,R}$
 satisfying
 $$\Gamma_{\varepsilon, R}(u_{\varepsilon,R})\in [M_0,M_1],\quad \|u_{\varepsilon,R}\|_{\varepsilon,R}\leq M_1.$$
 Moreover, $u_{\varepsilon,R}$ is a positive weak solution to
\begin{equation}\label{eq5.2}
-\Delta u+T_R(x,u)u= K(\varepsilon x)f_\varepsilon(x,u),\quad x\in \mathbb R^N,
\end{equation}
where
$$\begin{aligned}
T_R(x,u)=\widetilde V_\varepsilon(x)&+\overline V_\varepsilon (x)(\frac{1}{2}\eta_\varepsilon(x,|u|)|u|+\widehat\eta_\varepsilon(x,|u|))+\Big[\big(|x|-\frac R\varepsilon\big)^+\Big]^2.
\end{aligned}
$$
 \end{proposition}
 Clearly, the conclusion of Lemma \ref{lem 2.7} holds for $u_{\varepsilon,R}$ with a constant  $C>0$ independent of $\varepsilon\in (0,\varepsilon_0)$ and $R\geq R_0$.
 Therefore we have
 \begin{align}
 \sup_{\varepsilon\in (0,\varepsilon_0),R\geq R_0}&\|u_{\varepsilon,R}\|_{L^\infty(\mathbb R^N)}<\infty.\label{eq5.3}
 \end{align}
To describe more details on the localization of $u_{\varepsilon,R}$,
as in Section \ref{sec3}, we  investigate the ground state of the following functional
$$I_{a,b}(u)=\frac{1}{2}\int_{\mathbb R^N}|\nabla u|^2+(a+b) u^2
-b u^2\log u^2, \quad u\in H^1(\mathbb R^N),$$
where $a\in\mathbb R$ and $b>0$.
Denote
\begin{equation*}
m(a,b):=\inf_{u\in\mathcal N_{a,b}}I_{a,b}(u),
\end{equation*}
where $\mathcal N_{a,b}:=\set{u\in \mathcal D\setminus\{0\}| \int_{\mathbb R^N}|\nabla u|^2+a u^2-b u^2\log u^2=0}$.
It is easy to check that $m(a,b)$ is achieved by the unique
positive solution $U_{a,b}(x):=e^{\frac{a}{2b}}U(b^\frac12x)$ (up to translations in
$\mathbb R^N$) to

\begin{equation*}
\begin{cases}
-\Delta v+a v=b v\log v^2 \  \ &\text { in}\ \ \mathbb R^N,\\
v(x)\to 0 \ \ &\text { as}\ \ |x|\to\infty,
\end{cases}
\end{equation*}
and
$$m(a,b)=\max_{t\geq 0}I_{a,b}(tU_{a,b})=I_{a,b}(U_{a,b})=\frac b2|U_{a,b}|^2_2=\frac12e^{\frac ab}b^{-\frac N2+1}|U|_2^2.
$$
We note that
\begin{equation*}
m(V(z),K(z))>\frac{1}{2}P_0|U|_2^2
\quad \text{for}\quad z\in \overline\Omega\setminus \mathcal P.
\end{equation*}
By this and similar arguments to Lemma \ref{lem2.10} and Lemma \ref{lem 3.1}, we have
\begin{lemma}\label{lem 4.2} It holds that
$$\lim_{\varepsilon\to0}\sup_{R\geq R_0}\Gamma_{\varepsilon,R}(u_{\varepsilon,R})= \frac{1}{2}P_0|U|_2^2.$$
\end{lemma}

By  Lemma \ref{lem 4.2}, we can obtain
\begin{lemma}\label{lem 4.3}
For any $\delta>0$, there holds
$$\lim_{\varepsilon\to0}\sup_{R\geq R_0}\|u_{\varepsilon,R}\|_{L^\infty(\mathbb R^N\setminus (\mathcal P^\delta)_\varepsilon)}=0.$$
\end{lemma}

Then we can obtain the decay estimates
for $u_{\varepsilon,R}$.
\begin{proposition}\label{pro 4.4}
For each $\delta >0$, there exists $C,c>0$ such that
$$\sup_{R\geq R_0}|u_{\varepsilon,R}(x)|\leq  C\exp\big\{-c\varepsilon^{-\kappa}({\rm  dist}(x, (\mathcal P^\delta)_\varepsilon))^{2-\kappa}\big\}\quad \text{for}\ \varepsilon\in (0,\varepsilon_0)\ \text{and}\ x\in\mathbb R^N.$$
\end{proposition}
\begin{proof}
By \eqref{eq5.2}, for $\varepsilon>0$ small, similarly to Proposition \ref{pro 3.3} and by Lemma \ref{lem 4.3},
$w_{\varepsilon,R}:=|u_{\varepsilon,R}|$ satisfies
\begin{align}\label{eq5.4}
-\Delta w_{\varepsilon,R}+w_{\varepsilon,R}&\leq \sigma_2 w_{\varepsilon,R}\log w_{\varepsilon,R}^2,\quad &x\in&  B(0, R_0\varepsilon^{-1})\setminus (\mathcal P^{\delta/2})_\varepsilon,\\
-\Delta w_{\varepsilon,R}+w_{\varepsilon,R}&\leq \frac12\sigma_1\varepsilon^{-\kappa}|x|^{-\kappa} w_{\varepsilon,R}\log w_{\varepsilon,R}^2,\quad &x\in& \mathbb R^N\setminus B(0, R_0\varepsilon^{-1}).\label{5.5}
\end{align}
It follows from the comparison argument and \eqref{5.5} that
$$w_{\varepsilon,R}(x)\leq
\exp\Big\{-\frac{\sigma_1\varepsilon^{-\kappa}(\text { dist}(x, (\mathcal P^\delta)_\varepsilon))^{2-\kappa}}{2^{2-\kappa}(2-\kappa)^2}\Big\}\quad \text{for}\quad x\in \mathbb R^N\setminus B(0,2R_0\varepsilon^{-1}).
$$
On the other hand, for $x\in  B(0,2R_0\varepsilon^{-1})\setminus (\mathcal P^{\delta/2})_\varepsilon$,
$\sigma_2=\sigma_2(\varepsilon|x|)^\kappa |\varepsilon x|^{-\kappa}\geq \sigma_2(\delta/2)^\kappa |\varepsilon x|^{-\kappa}$.
Therefore, by \eqref{eq5.4} and \eqref{5.5}, we have
$$-\Delta w_{\varepsilon,R}+w_{\varepsilon,R}\leq \frac12 \sigma_3\varepsilon^{-\kappa}|x|^{-\kappa} w_{\varepsilon,R}\log w_{\varepsilon,R}^2,
\quad x\in  \mathbb R^N\setminus (\mathcal P^{\delta/2})_\varepsilon.
$$
with $\sigma_3:=\min\{2^{1-\kappa}\sigma_2\delta^\kappa,\sigma_1\}$.
Then there is $\varepsilon_\delta>0$ such that for
$\varepsilon\in (0,\varepsilon_\delta)$,
$$w_{\varepsilon,R}(x)\leq
\exp\Big\{-\frac{\sigma_3\varepsilon^{-\kappa}(\text { dist}(x, (\mathcal P^\delta)_\varepsilon))^{2-\kappa}}{(2-\kappa)^2}\Big\}\quad \text{for}\quad x\in  \mathbb R^N\setminus (\mathcal P^\delta)_\varepsilon.
$$
Then by \eqref{eq5.3}, the conclusion of the proposition holds in a similar way to Proposition \ref{pro 3.3}.
\end{proof}
Now we can complete the proof of Theorem \ref{th2}.
\begin{proof}[Proof of Theorem \ref{th2}]
By Proposition \ref{pro 4.4},
there are $C,c>0$ independent of $\varepsilon\in (0,\varepsilon_0)$ and $R\geq R_0$ such that
\begin{equation}\label{5.6}
|u_{\varepsilon,R}(x)|\leq  Ce^{-c\varepsilon^{-\kappa}|x|^{2-\kappa}}\quad \text{for}\quad x\in\mathbb R^N\setminus B(0,\Omega_\varepsilon),
\end{equation}
which implies
$$\eta_\varepsilon(x,|u_{\varepsilon,R}|)|u_{\varepsilon,R}|=0,\quad \widehat\eta_\varepsilon(x,|u_{\varepsilon,R}|)=1,\quad
f_\varepsilon(x, u_{\varepsilon,R})=u_{\varepsilon,R}\log u^2_{\varepsilon,R}.$$
Therefore, by \eqref{eq5.2}, $u_{\varepsilon,R}$ is a weak solution to
\begin{equation}\label{5.7}
	-\Delta u+V(\varepsilon x)u+\Big[\big(|x|-\frac R\varepsilon\big)^+\Big]^2u
=K(\varepsilon x)u\log u^2.
\end{equation}
On the other hand, by Proposition \ref{pro4.1},$\|u_{\varepsilon,R}\|_{\varepsilon}\leq \|u_{\varepsilon,R}\|_{\varepsilon,R}\leq M_1$. Therefore, up to a subsequence, as $R\to\infty$,
$u_{\varepsilon,R}\rightharpoonup u_\varepsilon$ in $H_\varepsilon$, for some
$u_\varepsilon\in H_\varepsilon$. By
\eqref{5.7} and Fatou's lemma, $u_\varepsilon\in \mathcal D$ and
it is a weak solution to \eqref{eq5.1}.
By \eqref{5.6},
$$
\lim_{R\to\infty}\int_{\mathbb R^N}\Big[\big(|x|-\frac R\varepsilon\big)^+\Big]^2u_{\varepsilon,R}^2=0.
$$
Then, by compact embedding from $H_\varepsilon$ to $L^p(B(0,R/\varepsilon))$, we can conclude $u_{\varepsilon,R}\to u_{\varepsilon}$ strongly in $H_\varepsilon$.
Note by
$$M_0\leq \lim_{R\to\infty}\Gamma_{\varepsilon,R}(u_{\varepsilon,R})=
\Gamma_{\varepsilon}(u_{\varepsilon})=
\frac12\int_{\mathbb R^N}K(\varepsilon x)u_\varepsilon^2
$$
that $u_\varepsilon\not \equiv 0$.
Then we have proved the existence of a solution to \eqref{1.4} and (i) of Theorem \ref{th2}. Property (ii) holds similarly
to Theorem \ref{th1.1}.
 \end{proof}

\appendix
\renewcommand{\theequation}{\thesection.\arabic{equation}}
\section{Some extensions}
\subsection{Singular potential}
 In this section, we consider the logarithmic equation
 \eqref{1.1} with potential function $V$ possessing a finite number
of singularities of at most logarithmic strength.
For equation \eqref{1.1}, assume
\begin{enumerate}
\item[(L1)] There exist $\ell(\in \mathbb N\setminus \{0\})$ distinct points
$\{z_j\}_{j=1}^\ell\subset \mathbb{R}^N$ such that $V\in C(\mathbb R^N\setminus \{z_j\}_{j=1}^\ell,\mathbb R)$  and for each $j=1,\cdots,\ell$,
\[
-\infty=\liminf_{x\to z_j} V(x)\leq \limsup_{x\to z_j}V(x)<\infty,\quad
\alpha_j:=
\limsup_{|x-z_j|\to 0} \frac{V(x)}{\log |x-z_j|^2}\in [0,\infty).
\]
In addition, there exists a $j_0 \in \{ 1, \ldots, \ell \}$ such that
	\[ \liminf_{x\to z_{j_0}} \frac{V(x)}{\log|x-z_{j_0}|^2} >0.
	\]
\end{enumerate}

We remark that in \cite{ITWZ}, Schr\"odinger equations with a more general type of nonlinearities
including the logarithmic one are investigated.
The potential function therein is assumed to satisfy (L1) but
possessing a lower bound at infinity and positive
 localized standing wave solutions are proved to exist
 concentrating at singular point $z_{j_0}$. As a generalization
 of this result to Schr\"odinger equations
 with potentials unbound below at infinity,
 we study the logarithmic equations \eqref{1.1} under the assumptions (V0) and (L1). We note that a typical example of potential function satisfying
 these assumptions is $\sum_{j=1}^\ell \alpha_j\log|x-z_j|^2-|x|^2$.
To describe the existence and asymptotic behaviors of solutions, we give the following theorem.
 \begin{theorem}
 Assume (V0) and (L1). Then there exists  $\varepsilon_0>0$
 such that for $\varepsilon\in (0,\varepsilon_0)$, and $j=1,\cdots,\ell$,
 \eqref{1.1} admits a positive solution $v_{\varepsilon}$ satisfying
 \begin{enumerate}
 \item[(i)]
 For any $\delta\in(0,1)$, there exist  $C>0$ such that
$$|v_{\varepsilon}(x)|\leq Ce^{-c\varepsilon^{-2}|x-z_{j_0}|^2}\ \ \text{for}\ \ |x-z_{j_0}|\geq \delta.$$
\item[(ii)] $\lim_{\varepsilon\to 0}\|v_{\varepsilon}\|_{L^\infty(\mathbb R^N)}=0$
and
$\lim_{\varepsilon\to 0}\varepsilon^{-\theta}\|v_{\varepsilon}\|_{L^\infty(\mathbb R^N)}=\infty$ for each $\theta>\alpha_{j_0}$.
\item[(iii)]
If we assume further that
\begin{enumerate}
\item[(L2)] There exists
$A_{j_0} \in \mathbb R$ such that
   \[
	   \lim_{x\to z_{j_0}}(V(x)-\alpha_{j_0}\log |x-z_{j_0}|^2)=A_{j_0},
   \]

\end{enumerate}
then for each sequence
	$\varepsilon_k \to 0$ there exists a subsequence (still denoted by $\varepsilon_k$ ) such that
		\[
			\varepsilon_k^{- \alpha_{j_0}   }v_{\varepsilon_k}(\varepsilon_k x+z_{j_0}) \to v_{j_0} \quad \text{as}\ k\to\infty
			\ \text{strongly in $H^1(\mathbb R^N)$}
		\]
	where $v_{j_0}$ is a positive ground state solution to
	\begin{equation}
		-\Delta v+\alpha_{j_0} \left(\log|x|^2 \right)v+A_{j_0}v= v\log v^2.
		\end{equation}
 \end{enumerate}

 \end{theorem}
\begin{proof}
We only sketch the proof for existence of a solution and its decay estimate (i).
Without loss of generality we assume
$j_0=1$, $z_{j_0}=z_1=0$ and $\alpha_{j_0}=\alpha_1$.
Fix $R_0>2\max\{|z_j|\}_{j=1}^\ell$.
Let $\widetilde V$ and $\Psi_\varepsilon$ be defined as in \eqref{tildev} and \eqref{Psi}.
Let $\alpha=\max\{\alpha_j\}_{j=1}^\ell$  and take any
$\theta\in(\alpha_1,2\alpha)$.

Then we can choose $\tau\in (0,\frac14\min_{i\neq j}\{|z_i-z_j|\})$ (set $\min_{i\neq j}\{|z_i-z_j|\}=\infty$ if $\ell=1$)
such that
\begin{equation}\label{6.2}
\begin{aligned}
\theta\log|x|^2\leq& V(x)\leq \beta\log|x|^2\quad \text{for any}\ x\in B(0,\tau),\\
2\alpha\log|x-z_j|^2\leq& V(x)\quad \text{for any}\ x\in B(z_j,\tau)\ \text{and}\ j=2,\cdots,\ell,
\end{aligned}
\end{equation}
where $\beta>0$ is a fixed constant such that $\beta<\liminf_{|x-z_1|\to 0} \frac{V(x)}{\log |x-z_1|^2}$.
Setting $u(x)=v(\varepsilon x)$, we solve the equation
\begin{equation*}
-\Delta u+V(\varepsilon x)u=u\log u^2 \  \ \text { in}\ \ \mathbb R^N.
\end{equation*}
Redefine $\chi_\varepsilon$ as the characteristic function of $\mathbb R^N\setminus B(0,\tau)$.
Following the idea of \cite{BW-1}, we introduce another penalization term.
Fix a function $W\in C^1(\mathbb R,\mathbb R)$ such that
\begin{equation*}
W'(s)\in[0,1] \quad \text{and}\quad
W(s)=
\begin{cases}
0,\ &s\leq \frac{1}{2},\\
s-1, \ &s\geq \frac{3}{2}.
\end{cases}
\end{equation*}
Set
\begin{equation}\label{w}
W_\varepsilon(s):=\varepsilon^{2\alpha} W(\varepsilon^{-2\alpha}s)
\end{equation}
and $Q_\varepsilon(u):=W_\varepsilon\big(\int_{\mathbb R^N}\varepsilon^{-6\alpha}\chi_\varepsilon(x) u^2\big)$. We have
$|W_\varepsilon(s)-W_\varepsilon'(s)s|\leq \frac32\varepsilon^{2\alpha}$ and hence $|Q_\varepsilon(u)-\frac{1}{2}Q_\varepsilon'(u)u|\leq \frac{3}{2}\varepsilon^{2\alpha}$  for all $u\in H_\varepsilon$
Redefine $\Gamma_\varepsilon :H_\varepsilon\to\mathbb R$ as
\begin{equation*}
\Gamma_{\varepsilon}(u)
=\frac{1}{2}\int_{\mathbb R^N}(|\nabla u|^2+\widetilde V_\varepsilon u^2)+\Psi_\varepsilon(u)+Q_\varepsilon(u)
-\frac{1}{2}\int_{\mathbb R^N}(u^2\log u^2-u^2)dx,
\end{equation*}
where
 $H_\varepsilon$ is redefined as the Hilbert space
 $$H_\varepsilon:=\Set{u\in H^1(\mathbb R^N) | \int_{\mathbb R^N}\widetilde V_\varepsilon(x)^+u^2dx<\infty},$$
with inner product $(u,v)_\varepsilon:=\int_{\mathbb R^N}\nabla u \nabla v+(1+\widetilde V_\varepsilon (x)^+)uv$ and norm $\|u\|_\varepsilon:=\sqrt{(u,u)_\varepsilon}$.
Through a similar argument to Lemma \ref{lem 2.5}, we can get for small $\varepsilon$,
\begin{align*}
\inf_{\|u\|_{\varepsilon}=r_0\varepsilon^\theta}\Gamma_\varepsilon(u)\geq M_0\varepsilon^{2\theta},&\quad
\inf_{\|u\|_{\varepsilon}\leq r_0\varepsilon^\theta}\Gamma_\varepsilon(u)>-1,\\
\sup_{t\geq0}\Gamma_\varepsilon(t\omega)\leq M_1\varepsilon^{2\beta},&
\quad \quad \sup_{t\geq t_0}\Gamma_\varepsilon(t\omega)<-2,
\end{align*}
where  $r_0, M_0, M_1, t_0$ are positive constants independent of $\varepsilon$ and $\omega\in C_0^\infty(\mathbb R^N)\setminus\{0\}$ is a fixed function.
Therefore $\Gamma_\varepsilon$ admits a mountain pass geometry. By the compactness of (PS) sequence, a critical point $u_\varepsilon$ to
$\Gamma_\varepsilon$ exists such that
$\varepsilon^{-\beta}\|u_\varepsilon\|_\varepsilon$  and $Q_\varepsilon(u_\varepsilon)$ are bounded
for small $\varepsilon>0$. We refer \cite{ITWZ} for details about these estimates.
Since for any $\delta \in (0,\tau/2)$, $\widetilde V_\varepsilon$ is bounded from below in
$\mathbb R^N\setminus \big(B(0, \frac12\delta\varepsilon^{-1})\bigcup_{j=2}^\ell B(z_j\varepsilon^{-1},\frac12\delta\varepsilon^{-1})\big)$, similar to Proposition \ref{pro 3.3},
we have
$$\lim_{\varepsilon\to 0}||u_\varepsilon||_{L^\infty (\mathbb R^N\setminus (B(0, \delta\varepsilon^{-1})\bigcup_{j=2}^\ell B(z_j\varepsilon^{-1},\delta\varepsilon^{-1})))}=0.
$$

Therefore, we can deduce that  there are $C,c>0$,
\begin{equation}\label{6.4}
|u_\varepsilon(x)|
\leq C\sum_{j=1}^\ell e^{- c|x-z_j\varepsilon^{-1}|^2}, \quad \text{for}\ x\in \mathbb R^N\setminus \Big(B(0, \delta\varepsilon^{-1})\bigcup_{j=2}^\ell B(z_j\varepsilon^{-1},\delta\varepsilon^{-1})\Big).
\end{equation}
On the other hand,
$\varepsilon^{-8\alpha}\int_{\mathbb R^N\setminus B(0, \varepsilon^{-1}\tau)}u_\varepsilon^2$ is bounded by the boundedness
of $Q_\varepsilon(u_\varepsilon)$.
For $x\in B(z_j\varepsilon^{-1},\tau\varepsilon^{-1})$, $j=2,\cdots,\ell$,
since $\widetilde V_\varepsilon=V_\varepsilon$ and $\overline V_\varepsilon=0$, we have
\begin{equation}\label{6.5}
	-\Delta(\varepsilon^{-3\alpha}u_\varepsilon) +(V_\varepsilon-3\alpha\log \varepsilon^2-\log (\varepsilon^{-3\alpha}u_\varepsilon)^2)(\varepsilon^{-3\alpha}u_\varepsilon)=0.
\end{equation}
By \eqref{6.2},
$[V_\varepsilon(x)-3\alpha\log \varepsilon^2-\log (\varepsilon^{-3\alpha}u_\varepsilon)^2]^-\leq 2\alpha(\log|x-z_j\varepsilon^{-1}|^2)^-+[\log (\varepsilon^{-3\alpha}u_\varepsilon)^2]^+$ for $x\in B(z_j\varepsilon^{-1},\tau\varepsilon^{-1})$.
Therefore, by the sub-solution  estimate in \cite{S},
$$\lim_{\varepsilon\to0}\|\varepsilon^{-3\alpha}u_\varepsilon\|_{L^\infty(\bigcup_{j=2}^\ell B(z_j\varepsilon^{-1},\frac12\tau\varepsilon^{-1}))}
=0.
$$
Then, by \eqref{6.2} and \eqref{6.5}
\begin{equation}\label{6.6}
-\Delta|\varepsilon^{-3\alpha}u_\varepsilon| +(2\alpha\log|x-z_j\varepsilon^{-1}|^2-\alpha \log \varepsilon^2)|\varepsilon^{-3\alpha}u_\varepsilon|\leq 0
.
\end{equation}
Next we consider
$$h(s)=\begin{cases}
 s^2\log s^2+1\quad &\text{if}\quad s\in[0,e^{-\frac12}],\cr
1-e^{-1}\quad &\text{if}\quad s\in(e^{-\frac12},\infty).
\end{cases}
$$
Note that $\psi_\varepsilon(x)= h(\alpha^{\frac12}|x-z_j\varepsilon^{-1}|)\geq 1-e^{-1}>0$ is a weak $H^1_{loc}(\mathbb R^N)$
 solution to
\begin{equation}\label{6.7}
	-\Delta\psi_\varepsilon +(2\alpha\log|x-z_j\varepsilon^{-1}|^2-\alpha \log \varepsilon^2)\psi_\varepsilon\geq 0.
\end{equation}
As a result of \eqref{6.6} and \eqref{6.7}, in $B(z_j\varepsilon^{-1},\frac12\tau\varepsilon^{-1})$,   we have
\begin{equation*}
	-\Delta w_\varepsilon +
(2\alpha\log|x-z_j\varepsilon^{-1}|^2-\alpha \log \varepsilon^2)w_\varepsilon\leq 0,
\end{equation*}
where $w_\varepsilon:=|\varepsilon^{-3\alpha}u_\varepsilon|-\varepsilon^{-4\alpha}e^{-c\varepsilon^{-2}}\psi_\varepsilon$.
By \eqref{6.4}, for $\varepsilon$ small,  $w_\varepsilon^+\in H_0^1(B(z_j\varepsilon^{-1},\frac12\tau\varepsilon^{-1}))$.
Noting that for small $\varepsilon$, the operator $-\Delta+
(2\alpha\log|x-z_j\varepsilon^{-1}|^2-\alpha \log \varepsilon^2)$ is positively definite on $H_0^1(B(z_j\varepsilon^{-1},\frac12\tau\varepsilon^{-1}))$,
we have
$w_\varepsilon^+=0$ and thus
$$\|u_\varepsilon\|_{L^\infty(B(z_j\varepsilon^{-1},\frac12\tau\varepsilon^{-1}))}\leq \varepsilon^{-\alpha}e^{-c\varepsilon^{-2}}.
$$
Together with \eqref{6.4}, we have
\begin{equation*}
|u_\varepsilon(x)|
\leq Ce^{- c|x|^2}, \quad \text{for}\ x\in \mathbb R^N\setminus B(0, \delta\varepsilon^{-1}).
\end{equation*}
As a result, $u_\varepsilon$ is the solution to the original problem.
\end{proof}

\subsection{Multiple solutions}
We consider the existence of multiple solutions for \eqref{1.1} and assume that
\begin{enumerate}
\item[(V4)] $V\in C^1(\mathbb R^N,\mathbb R)$  and there is a bounded domain
  $\Omega\subset {\mathbb{R}}^{N}$ with smooth boundary  such that
   \begin{equation*}
 \nabla V(x)\cdot \vec{n}(x)>0,\ \ x\in\partial\Omega,
 \end{equation*}
 where $\vec n (x)$ denotes the outward unit normal vector to $\partial \Omega$ at $x$.
\end{enumerate}
 \begin{theorem}
Let (V0) and (V4) hold. Then for any positive integer $k$, there exists $\varepsilon_k>0$ such that for all $\varepsilon\in(0,\varepsilon_k)$, equation \eqref{1.1} has $k$ pairs of nontrivial solution $\pm v_{\varepsilon,i}$, $i=1,2,...,k$. In addition,
  for each $\delta\in(0,1)$, $i=1,2,...,k,$ there is $C=C(\delta, i),c>0$ such that
$$|v_{\varepsilon,i}(x)|\leq Ce^{-c \varepsilon^{-2}(\text{dist}(x,\mathcal V^{\delta}))^{2}}.$$
 \end{theorem}

To sketch the proof, let $\zeta \in C^\infty(\mathbb R,[0,1])$ be a cut-off function such that   $\zeta'(t) \geq0$ for every $t \in \mathbb R$, $\zeta(t) =0$ if $t \leq0$, $0<\zeta(t) <1$ if $0<t < 1$ and $\zeta(t) =1$ if $t \geq 1$.
We set
$$\widetilde\chi_\varepsilon(x)=
\begin{cases}
0,&x\in\Omega_\varepsilon;\\
\varepsilon^{-6}\zeta\big(\text { dist}(x,\Omega_\varepsilon)\big),&x\not\in \Omega_\varepsilon.
\end{cases}
$$
Then for $W_\varepsilon$ given in \eqref{w} with $\alpha =1$, we define
$$\widetilde Q_\varepsilon(u):=W_\varepsilon\Big(\int_{\mathbb R^N}\widetilde\chi_\varepsilon u^2\Big),$$
and give the modified functional: $H_\varepsilon\to\mathbb R$
\begin{equation*}
\Gamma_{\varepsilon}(u)=\frac{1}{2}\|u\|_\varepsilon^2+\Psi_\varepsilon(u)+\widetilde Q_\varepsilon(u)
-\frac{1}{2}\int_{\mathbb R^N}(u^2\log u^2-u^2)dx,
\end{equation*}
where $H_\varepsilon$ and $\Psi_\varepsilon$ are defined in \eqref{Psi} and \eqref{H}.
It is easy to check that  the results in Lemma \ref{lem 2.3}--Lemma \ref{lem 2.5} also hold for the newly defined $\Gamma_\varepsilon$.
Similar to \cite{Wang-Zhang-1},
 let $\{e_i\}\subset C_0^\infty(B(0,1))$ be an orthonormal basis of $H_0^1(B(0,1))$. Setting $E_k:=\text { span}\{e_1,...,e_k\}$ for any integer $k>0$, then there exist $R_k>0$ and $M_k>0$ such that
 \begin{equation}\label{6.8}
 \sup_{u\in E_k, \|u\|_\varepsilon\geq R_k}\Gamma_\varepsilon(u)<-2\quad \text{and}\quad
 \sup_{u\in E_k, \|u\|_\varepsilon\leq R_k}\Gamma_\varepsilon(u)\leq M_k.
 \end{equation}
For each $k$ and $\varepsilon\in(0,\varepsilon_k)$, according to Lemma \ref{lem 2.5} (ii), \eqref{6.8} and Corollary \ref{cor 2.4},
 we may apply the symmetric mountain-pass theorem \cite{Ra}
 to $\Gamma_\varepsilon$ and obtain $k$ pairs of solutions $\pm u_{\varepsilon,l}$, $l=1,2,...,k$ with
\begin{equation*}
\Gamma_\varepsilon(u_{\varepsilon,l})\in[M_0,M_k],\ \ l=1,2,...,k,\ \ \varepsilon\in(0,\varepsilon_k).
\end{equation*}

Then, we can localize these critical points by a local Pohozaev identity. See \cite[Section 4]{Wang2017} or \cite{Wang-Zhang-2}  for a detailed procedure.
At last, one can recover the original problem by showing the decay property of these critical points.


\end{document}